\documentclass[a4paper, reqno, english]{smfart}

\usepackage{amsfonts, amsthm, amsmath, amscd, amssymb}
\usepackage{graphicx}
\RequirePackage{mathrsfs} \let\mathcal\mathscr

\numberwithin{equation}{section}

\renewcommand{\SS}{\mathfrak{S}}

\renewcommand{\d}{\mathrm{d}}
\newcommand{\x}{\ma{x}}
\newcommand{\w}{\ma{w}}
\newcommand{\y}{\ma{y}}
\newcommand{\z}{\ma{z}}

\renewcommand{\t}{\ma{t}}

\renewcommand{\rho}{\varrho}

\newcommand{\M}{\mathfrak{M}}
\newcommand{\m}{\mathfrak{m}}

\newcommand{\CC}{\mathcal{C}}

\newcommand{\R}{\mathbb{R}}
\newcommand{\F}{\mathbb{F}}

\newcommand{\Z}{\mathbb{Z}}
\newcommand{\N}{\mathbb{N}}
\newcommand{\Q}{\mathbb{Q}}

\DeclareMathOperator{\Mod}{mod}
\renewcommand{\bmod}[1]{\,(\Mod{ #1})}

\newcommand{\ma}{\mathbf}

\newcommand{\mcal}{\mathcal}

\renewcommand{\le}{\leqslant}
\renewcommand{\ge}{\geqslant}
\renewcommand{\leq}{\leqslant}
\renewcommand{\geq}{\geqslant}

\newcommand{\ben}{\begin{enumerate}}
\newcommand{\een}{\end{enumerate}}
\newcommand{\eit}{\begin{itemize}}
\newcommand{\beq}{\begin{equation}}
\newcommand{\eeq}{\end{equation}}

\newcommand{\ve}{\varepsilon}

\newcommand{\al}{\alpha}
\newcommand{\D}{\Delta}
\newcommand{\del}{\delta}

\newcommand{\be}{\beta}
\newcommand{\la}{\lambda}

\newcommand{\lab}{\label}

\newtheorem{thm}{Theorem}
\newtheorem*{thm*}{Theorem}
\newtheorem{lemma}{Lemma}
\newtheorem{pro}{Proposition}

\newtheorem*{cor*}{Corollary}

\DeclareMathOperator{\rank}{rank}

\DeclareMathOperator{\meas}{meas}

\newcommand{\hcf}{\mathrm{gcd}}

\renewcommand{\L}{\mathcal{L}}

\newcommand{\B}{\mathcal{B}}

\newcommand{\XX}{\mathcal{X}} 

\theoremstyle{definition} 
\newtheorem*{ack}{Acknowledgement}
\newtheorem*{notat}{Notation}

\begin{document}

\title{Least zero of a cubic form}

\author{T.D. Browning}
\address{School of Mathematics\\
University of Bristol\\ Bristol BS8 1TW\\ United Kingdom}
\email{t.d.browning@bristol.ac.uk}

\author{R. Dietmann}
\address{Institut f\"ur Algebra und Zahlentheorie\\
Lehrstuhl~f\"ur Zahlentheorie\\
Pfaffenwaldring 57\\ D-70569 Stuttgart\\ Germany}
\email{dietmarr@mathematik.uni-stuttgart.de}

\author{P.D.T.A. Elliott}
\address{Department of Mathematics\\
University of Colorado at Boulder\\
Campus Box 395 Boulder\\ CO 80309-0395\\ USA}
\email{pdtae@euclid.colorado.edu}

\dedicatory{
\begin{center}In memory of H. Davenport\end{center}}

\date{\today}

\begin{abstract}
An effective search bound is established for the least non-trivial
integer zero of an arbitrary cubic form 
$C\in \Z[X_1,\ldots,X_n]$, provided that $n\geq 17$.
\end{abstract}


\maketitle
\tableofcontents

\section{Introduction}

Let $n \geq 3$ and let $F\in \Z[X_1,\ldots,X_n]$  
be an indefinite form of degree $d\geq 2$, with coefficients of
maximum modulus $\|F\|$ and greatest common divisor $1$.
It is very natural to try and ascertain procedures for determining
whether or not the equation $F=0$ is soluble in integers. One such
approach involves providing an effective upper bound for the smallest
positive integer $\la$  with the property that when
there is a non-zero solution $\x=(x_1,\ldots,x_n)\in\Z^n$
to the equation $F=0$, so there is such a solution with $
\max_{1\leq i \leq n}|x_i|  \leq \la$. 
Let us denote this quantity by $\Lambda_n(F)$ when it exists.

When $d=1$ the problem is straightforward, and it follows from
Siegel's lemma that $\Lambda_n(F)\leq n^{\frac{1}{n-1}}\|F\|^{\frac{1}{n-1}}$.  
For polynomials of higher degree the problem has received the most attention
in the case $d=2$ of quadratic forms $F=Q$. There
is a well-known result due to Cassels \cite{cassels}, which shows that  
$$
\Lambda_n(Q) \leq c_n \|Q\|^{\frac{n-1}{2}},
$$
with a completely explicit value of $c_n$.  
Although the exponent of $\|Q\|$ is known to be best possible in general, 
recent joint work of Browning and Dietmann \cite{qleast} demonstrates that one
can do much better for generic quadratic forms.  
It is interesting to remark that 
Cassels' bound played an important r\^ole in the work of Birch
and Davenport \cite{BD} on the solubility in integers of Diophantine
inequalities
$
|Q(x_1, \ldots, x_n)|<1,
$
for suitable quadratic forms $Q$ defined over $\R$.

The situation for forms of degree $d=3$ is far less
satisfactory, and a proper analogue of Cassels'
result for quadratic forms remains a distant prospect. 
Aside from the intrinsic interest of this problem, such a
bound would be very desirable in the context of 
cubic Diophantine inequalities. 

Let us record some of the progress
that has been made for cubic forms $F=C$. 
Suppose first that the form is diagonal and has $7$ variables, 
with coefficients $A_1,\ldots,A_7$. Then
Li \cite{li} has shown that there is a non-trivial
integral solution with 
$$
\sum_{i=1}^7 |A_ix_i^3|\leq c|A_1\cdots A_7|^{14},$$ 
for some absolute constant
$c>0$.    In particular it easily follows from this result that 
$
\Lambda_n(C) \leq c\|C\|^{\frac{95}{3}},
$
for any diagonal cubic form in $n\geq 7$ variables.

For general cubic forms one of the few results in the literature is due to
Pitman \cite{pitman}. For any $\ve>0$, Pitman establishes the
existence of constants $N_\ve$ and $c_{n,\ve}>0$ such that
$$
\Lambda_n(C) \leq c_{n,\ve} \|C\|^{\frac{25}{6}+\ve},
$$
whenever $n\geq N_\ve$. One notes that the exponent
of $\|C\|$ is independent of $n$, unlike the situation for quadratic
forms discussed above. However, the number of variables needed to make
this argument work is extremely large. This loss is due to the
wasteful nature of the proof, in which a diagonalisation process is
applied to reduce the problem to bounding $\Lambda_n(C)$ for a
diagonal cubic form.  Still working with cubic forms in many 
variables, it has been shown by Schmidt
\cite[Theorem 2]{schmidt} that 
Pitman's estimate is valid with the exponent $\frac{25}{6}$ replaced by
a function $e_1(n)$ that tends to $0$ as $n\rightarrow \infty$.

At the expense of a much weaker exponent of $\|C\|$,
it is nonetheless possible to produce estimates for $\Lambda_n(C)$
when $n$ is as small as $17$, by avoiding the use of diagonalisation
arguments.  This is the point of view adopted by Elliott \cite{elliott}
in his Ph.D. thesis, where it is shown that there exists
a constant $c_n>0$ such that 
$$
\Lambda_n(C)\leq c_n \|C\|^{e_2(n)},
$$
for $n\geq 17$, where
$$
e_2(n)=\begin{cases}
\{1+\frac{3}{4}(\frac{n+24}{n-16})(\frac{11n(n-1)}{n-9}+10n-13)\}\{
\frac{(2n+3)^2}{8}-n-1\}, &
\mbox{if $n\leq 30$,}\\
n\{1+\frac{3}{4}(\frac{3n+120}{3n-80})(\frac{11n(n-1)}{n-9}+10n-13)\}, &
\mbox{if $n\geq 30$.}
\end{cases}
$$
Taking $n=17$ one finds that $e_2(17)=
2500417$.
In later work, seemingly unaware of Elliott's thesis, 
Lloyd \cite{lloyd} succeeded in showing that 
$$
\Lambda_{17}(C)\leq c \|C\|^{8\times 10^8},
$$
for some absolute constant $c>0$ and any non-singular cubic form in
$17$ variables. Thus Lloyd's exponent is worse than that obtained by
Elliott and has the defect of only applying to non-singular cubic
forms.

The primary aim of this paper is to stimulate further interest in this
and allied problems. Our main achievement will be a 
sharper upper bound for $\Lambda_n(C)$ when $n \geq 17$. The following
result deals with cubic forms $C$ for which the corresponding hypersurface $C=0$
has a suitably restricted singular locus.

\begin{thm}\label{main-ns}
Let $C\in\Z[X_1,\ldots,X_n]$ be a cubic form, with $n\geq
17$, defining a hypersurface with at most isolated ordinary
singularities. Then for any $\ve>0$ there exists a
constant $c_{n,\ve}>0$ such that
$$
\Lambda_n(C)\leq c_{n,\ve} \|C\|^{e_3(n)+\ve},
$$
where
\begin{equation}
  \label{eq:bd}
e_3(n)
=
\begin{cases}
\frac{
22n^3+107n^2-597n-432}{(n-2)(n-16)(n-9)},
&\mbox{if $n\leq 20$,}\\
\frac{n^4+125n^3+1518n^2-7236n-4320}{32(n-2)(n-9)}, &\mbox{if $n> 20$.}
\end{cases}
\end{equation}
\end{thm}

\begin{figure}[tbp]
\begin{center}
\includegraphics[angle=-90, scale=0.35]{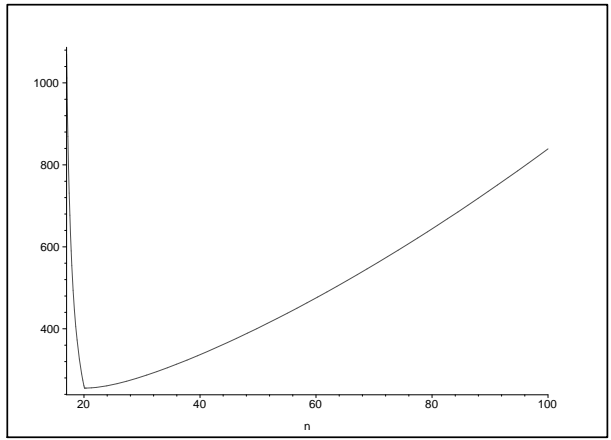}
\end{center}
\begin{center}
\caption{The function $e_3(n)$ for $17\leq n\leq 100$}
\label{fig}
\end{center}
\end{figure}

The constant $c_{n,\ve}$ in Theorem \ref{main-ns} is effectively computable, a
feature shared by all the implied constants in this work.
In Figure \ref{fig} we have graphed the function $e_3(n)$ for small values of $n$. 
Taking $n=17$ and $\ve$ sufficiently small, one concludes from 
Theorem \ref{main-ns} that 
$$
  \Lambda_{17}(C)\leq c \|C\|^{1071}
$$
for an absolute constant $c>0$, provided that the hypersurface $C=0$ 
is non-singular or contains isolated ordinary singularities.
When $n= 20$ this exponent can be improved to $261$.
Theorem \ref{main-ns} provides a palpable improvement over the earlier
works of Elliott or Lloyd discussed above. In fact
$e_3(n)\sim \frac{1}{32}n^2$, as $n\rightarrow \infty$,
whereas   $e_2(n)\sim \frac{63}{4} n^2$.

It is natural to ask what can be said about arbitrary cubic forms. 
At the expense of a weaker exponent, we answer this in the following result.

\begin{thm}\label{main}
Let $C\in\Z[X_1,\ldots,X_n]$ be a cubic form, with $n\geq
17$. Then there exists an absolute constant $c>0$ such that
$$
  \Lambda_{n}(C)\leq c \|C\|^{360000}.
$$
\end{thm}

All of the bounds for $\Lambda_n(C)$ that we have discussed so
far are based on applications of the Hardy--Littlewood circle method, 
and our approach is no exception. The informed reader will
notice that the restriction to cubic forms in at least $17$ variables is at
odds with current understanding of rational points on cubic
hypersurfaces. Indeed, by employing the machinery of Hooley
\cite{hooley2} one ought to be able to handle cubic forms
that define hypersurfaces with at most isolated ordinary
singularities in only $10$ variables. Similarly, for general
cubic forms, the work of Heath-Brown \cite{14} should allow a
reduction from $n\geq 17$ to $n\geq 14$ in Theorem \ref{main}. Both of these improvements
would be at the expense of considerable extra labour, however.
In the present investigation we have decided
to place the emphasis on brevity of exposition, and it is hoped that
the deficiency alluded to above will be taken in the light of this fact.

The relative strength of our results arises from a more sophisticated
treatment of the 
major arcs and of the positivity of the singular series. 
The latter phase of the argument hinges upon good
lower bounds for the quantity
\begin{equation}
  \label{eq:rho-p}
\varrho(p^k)=
\#\{\x\in(\Z/p^k\Z)^n: C(\x)\equiv
0 \bmod{p^k}\},
\end{equation}
for prime powers $p^k$.  Taking $k=1$, the problem reduces to estimating the number of
points on the cubic hypersurface $C=0$ over $\F_p$. 
We will seek estimates of the form 
$\rho(p)=p^{n-1}+O(p^{n-\theta})$,
where the implied constant depends at
most on $n$.  When $C$ is
non-singular modulo $p$ the work of Deligne shows that 
$\theta=\frac{n}{2}$ is permissible.  For general $C$ we can take $\theta=\frac{1}{2}$
by the Lang--Weil
estimate, provided that $C$ is absolutely irreducible modulo $p$.
This is not good enough for our purposes however. We circumvent this
difficulty with the following result.

\begin{thm}
\label{enterprise}
Let $C \in \F_p[X_1, \ldots, X_n]$ 
for $n\ge 10$. Suppose further that $C$ is non-degenerate. Then
\begin{equation}
\label{kirk}
  \varrho(p) \ge p^{n-1} + O(p^{n-2}),
\end{equation} 
where the implied constant depends at most on $n$.
\end{thm}

The investigation of least non-trivial zeros of cubic forms is
currently enjoying a resurgence of interest.
The problem is most intriguing in the case $n=4$ of surfaces. 
As highlighted by Swinnerton-Dyer \cite[Question 15]{swd-prog}, 
a real milestone in this domain would be to discover whether one
could estimate $\Lambda_4(C)$ effectively when $C$ is diagonal. 
Elsenhans and Jahnel \cite{e-j1} 
have gone further, based on numerical calculations, 
asking whether one can expect inequalities of the sort
$$
\Lambda_4(C)\leq  \frac{c_\ve}{\tau(C)^{1+\ve}},
$$
where $\tau(C)$ is a Tamagawa number
associated to the corresponding cubic surface.

It would be interesting to explore whether the ideas in this paper
could be adapted to handle non-singular forms of degree exceeding $3$.
This would be in complete analogy to the extension by Birch
\cite{birch} to higher degree of Davenport's \cite{dav-16} treatment
of cubic forms.
It is easily checked that the treatment of minor and major arcs 
goes through with little
alteration. The main obstacle appears to be achieving effective
lower bounds for the singular integral and
singular series, respectively.

\begin{notat}
  All of the implied constants in our work will be allowed to depend on $n$
and $\varepsilon$, with any further dependence being made completely
explicit. We will adhere to common practice and allow $\ve$ to   
take different values at different parts of the argument, but we shall
always assume it to be very small.
Throughout the remainder of the paper we write
$$
M=\|C\|,
$$
for the height of the cubic form $C$ that is under scrutiny, and
$|\x|$ for the norm $\max_{1\leq i \leq n}|x_i|$ of any vector $\x \in
\R^n$.
\end{notat}

\begin{ack}
While working on this paper the first author was supported by EPSRC
grant number \texttt{EP/E053262/1}. 
\end{ack}

\section{Skeletal proof of the theorems}

We write our cubic form in the shape
\begin{equation}
  \label{eq:C}
C(x_1,\ldots,x_n)=\sum_{i,j,k}c_{ijk}x_ix_jx_k,
\end{equation}
in which the coefficients $c_{ijk}\in \Z$ are symmetric in the indices $i,j,k$.
According to our hypothesis concerning notation the modulus of
any coefficient $c_{ijk}$ is bounded by $M$. We claim that it will suffice to
proceed under the assumption that $c_{111}$ is positive, with
\begin{equation}
  \label{eq:anna}
  c_{111} \gg M,
\end{equation}
for an absolute implied constant. 

Now it is already plain that there is no loss of generality in
assuming that one of $|c_{111}|, |c_{112}|, |c_{123}|$ is at least
$M$. Suppose that $|c_{123}|\geq M$ and for $\sigma \in \{-1,1\}$
consider the unimodular transformation
$$
x_i\mapsto \begin{cases}
y_i, &\mbox{if $1\leq i \leq n$ and $i\neq 3$,}\\
\sigma y_1+y_3, &\mbox{if $i= 3$.}
\end{cases}
$$
This produces a new cubic form $C'(y_1,\ldots,y_n)$ such that $\|C'\|\ll M$, with
integer coefficients $c_{ijk}'(\sigma)$. In particular we
have
$$
|c_{112}'(\sigma)|=|c_{112}+2\sigma
c_{123}+c_{233}|=|2c_{123}|+|c_{112}+c_{233}|\geq 2M,
$$
on choosing $\sigma$ to be the sign of $c_{123}^{-1}(c_{112}+c_{233})$.

Hence it suffices to assume that $|c_{112}|\geq M$ above. We now
carry out the unimodular transformations
$$
x_i\mapsto \begin{cases}
y_i, &\mbox{if $1\leq i \leq n$ and $i\neq 2$,}\\
\sigma y_1+y_2, &\mbox{if $i= 2$,}
\end{cases}
$$
for $\sigma \in \{-1,1\}$. 
If $|c_{111}'(\sigma)|\geq \frac{M}{2}$ for either choice of $\sigma$ then we
will be done. Alternatively, we have 
$
\frac{M}{2}>|c_{111}'(\sigma)|=
|c_{111}+\sigma
c_{112}+c_{122}+\sigma c_{222}|,
$
for $\sigma=\in \{-1,1\}$. Choosing $\sigma$ to be the sign of
$c_{111}+c_{122}$, we deduce that $|c_{112}+c_{222}|<\frac{M}{2}$, whence
$$
|c_{222}|\geq |-c_{112}|-|c_{112}+c_{222}|\geq \frac{M}{2}.
$$
On permuting the variables we can therefore assume that
\eqref{eq:anna} holds in this case too. Note that the reduction from 
negative $c_{111}$ to positive $c_{111}$ is trivially achieved by
multiplying the equation through by $-1$.

The basic idea behind the proof of Theorems \ref{main-ns} and \ref{main} 
is very simple. Given a
suitable bounded region $\mcal{B}\subset \R^n$ we wish to 
establish an asymptotic formula for the counting function
$$
N(P)=N_{\mcal{B}}(P;C)=\sum_{\substack{\x\in\Z^n\cap P\mcal{B}\\ C(\x)=0}}1,
$$
as $P\rightarrow \infty$, uniformly in the coefficients of $C$.  
Here, $P\mcal{B}=\{P\x: \x\in \mcal{B}\}$.
One then obtains explicit bounds on the size of $P$ needed
to ensure that the main term dominates the error term in this
estimate. For such $P$ we will thus have 
$N(P)>0,$
which then yields an upper bound for $\Lambda_n(C)$.

As indicated in the introduction we plan to use the Hardy--Littlewood
circle method to estimate $N(P)$, based on  the argument
developed by Davenport \cite{dav-16}. 
For given $\z\in \R^n$ and $\rho\in (0,1)$, let 
$$
\mcal{B}=\mcal{B}(\z;\rho)=\prod_{i=1}^n[z_i-\rho, z_i+\rho].
$$
This is the box that we will work with.
Note that
$\meas(\mcal{B})=2^n\rho^n$. 
The choice of $\z$ and $\rho$ will be made in due course, but we
record now that 
\begin{equation}
  \label{eq:hope'}
M^{-2-\frac{5}{n-2}}\ll  \rho\ll M^{-2-\frac{5}{n-2}}, 
\quad |\z|\ll M^{\frac{1}{n-2}}.
\end{equation}

Recall the definition \eqref{eq:C} of the cubic form $C$. 
We define an $n\times n$ matrix $M(\x)$ by taking
its entries to be 
\begin{equation}\label{Mdef}
M(\x)_{jk}=\sum_{1\leq i\leq n} c_{ijk}x_i.
\end{equation}
Let $\psi\in \R_{\geq 1}\cup\{\infty\}$. 
If one writes
$r(\x)=\rank(M(\x))$, then we shall say that $C$ is
``$\psi$-good'' if for any $\ve>0$ the estimate
\begin{equation}
  \label{eq:altA'}
\#\{\x\in\Z^n: |\x|\leq H,\,r(\x)=r\}\ll H^{r+\ve}
\end{equation}
holds for each $0\leq r\le n$ and any $H$ in the range
$
1\leq H \leq M^{\psi}.
$
Although we will not make this restriction yet, it turns out that our
work is optimised by taking $\psi=\infty$ for the proof of Theorem
\ref{main-ns} and $\psi=2646$
for the proof of Theorem \ref{main}.

An $\infty$-good form is one for which \eqref{eq:altA'}  holds
for each $0\leq r\leq n$ and $H\geq 1$. 
It follows from  work of Hooley \cite[Lemma~28]{hooley2} that cubic forms 
defining a hypersurface with at most isolated ordinary
singularities are all $\infty$-good.
The following result, which will be proved in \S \ref{s:geom}, 
 handles the possibility that $C$ fails to be
$\psi$-good, for a given choice of $\psi$.

\begin{pro}\label{p:altA}
Let $n\geq 3$ and let $\psi\geq 1$. Then either $C$ is
$\psi$-good or 
\begin{equation}\label{altA}
\Lambda_n(C)\ll 
M^{\frac{n^2}{2}-1+\frac{\psi n(n-1)}{2}}.
\end{equation}
\end{pro}

By our remarks above we may assume that the cubic forms considered in
Theorem~\ref{main-ns}  are $\infty$-good.
For a $\psi$-good cubic form $C$ we now desire an estimate for $N(P)$
in which the implied constant is completely 
uniform in $\rho, \z$ and in the coefficients of $C$. 
Our starting point is the 
identity 
$
N(P)=\int_0^1 S(\al)\d\al,
$
where
$$
S(\al)=
\sum_{\x\in\Z^n\cap P\mcal{B}}e(\al C(\x)).
$$
We will always assume that $\rho P\geq 1$, so that this sum is non-trivial.
Let $P_0\geq 1$. 
We will take as major arcs 
$$
\M(a,q)=\Big[\frac{a}{q}-
\frac{P_0}{M\rho^3 P^{3}}, \frac{a}{q}+
\frac{P_0}{M\rho^3 P^{3}}
\Big],
$$
for given coprime integers $a,q$ such that $0\leq a< q\leq P_0$. The
full set of major arcs is
$$
\M=\bigcup_{q\leq P_0}\bigcup_{\substack{a=0\\\hcf(a,q)=1}}^{q-1} \M(a,q),
$$
and the corresponding set of minor arcs is $\m=[0,1]\setminus \M$,
defined modulo $1$. 
Our work will be optimised by taking 
\begin{equation}
  \label{eq:P0}
P_0=\big(M^{\frac{n}{8}}\rho P\big)^{\frac{8}{n+16}}.
\end{equation}
It is clear that $P_0\geq 1$.
Note, furthermore,  that the union of major arcs will be disjoint provided that
$  P_0\ll M^{\frac{1}{3}}\rho P.
$
Substituting in our choice of $P_0$, we see that this holds if and only if
$\rho P\gg M^{\frac{2(n-8)}{3(n+8)}}$, which follows from the
assumption that
\begin{equation}
  \label{eq:**}
P\gg M^{15},
\end{equation}
for $\rho$ satisfying \eqref{eq:hope'} and $n \geq 17$.
Moreover, under this assumption it follows that
\begin{equation}
  \label{eq:hick}
P_0\ll 
M^{\frac{n}{n+16}}
P^{\frac{8}{n+16}}\ll M^{-4}P^2.
\end{equation}
This will prove useful shortly.

The truncated singular series for our counting problem is given by 
\begin{equation}
  \label{eq:truncS}
\SS(R)=\sum_{q\leq R}\sum_{\substack{0\leq a<q\\ \hcf(a,q)=1}}
q^{-n}\sum_{\ma{r}\bmod q} e_q(aC(\ma{r})),
\end{equation}
for any $R\geq 1$. The assumption that our cubic forms are $\psi$-good
is rather weak when $\psi< \infty$ and we cannot hope to
establish the usual Hardy--Littlewood asymptotic formula for the full
class of $\psi$-good cubic forms. In particular, the  singular
series $\SS=\lim_{R\rightarrow \infty}\SS(R)$ may fail to converge 
in general and we will be forced to work with the
truncated series instead. 
The following two results handle the
contribution from the minor arcs and major arcs, and will be
established in \S \ref{s:minor} and \S
\ref{s:major}, respectively. 

\begin{pro}\label{minor}
Let $\ve>0$ and assume that $n\geq 17$. Assume that $C$ is $\psi$-good
and that $\rho P\geq 1$. Then we have
$$
\int_{\m} |S(\al)|\d\al 
\ll 
M^{-\frac{\psi n}{4}}
 \rho^n P^{n+\ve} +
M^{\frac{n}{8}-1} \rho^{n-3} P^{n-3+\ve}P_0^{2-\frac{n}{8}}.
$$
\end{pro}

\begin{pro}\label{major}
Let $\ve>0$ and assume that $n\geq 17$. 
Assume that $C$ is $\psi$-good
and \eqref{eq:**} holds. Then there exists a positive constant $\mathfrak{I}$ satisfying 
$\mathfrak{I}\gg \rho^{n-1}M^{-1-\frac{2}{n-2}}$ such that 
\begin{align*}
\int_{\M}S(\al)\d \al
=\SS(P_0) \mathfrak{I}P^{n-3} &+
O\big(M^{-1}\rho^{n-4}P^{n-4}P_0^4\big)\\
&+O\big(
M^{\frac{n}{8}+\frac{7}{2}+\frac{11}{n-2}}
\rho^{n-\frac{1}{2}} 
P^{n-3}P_0^{-\frac{1}{2}} \big)\\
&+O\big(
M^{6-\frac{\psi n}{4}}
\rho^{n} 
P^{n-3}P_0^{\frac{3}{2}+\ve} \big).
\end{align*}
\end{pro}

In our application we will take any $\rho$ in the range
\eqref{eq:hope'}, under which assumption 
we have 
\begin{equation}
  \label{eq:lowerI}
\mathfrak{I}\gg M^{-\frac{2n^2-1}{n-2}}
\end{equation}
in Proposition \ref{major}. Furthermore, 
under \eqref{eq:hick}, it follows that the third error term in
this result is
$
O(M^{-\frac{\psi n}{4}}\rho^n P^{n+\ve}).
$
Making the choice \eqref{eq:P0} for $P_0$, 
Propositions~\ref{minor} and  \ref{major} combine to give the
following result.

\begin{pro}\label{major'}
Let $\ve>0$ and assume that $n\geq 17$. 
Assume that $C$ is $\psi$-good
and \eqref{eq:**} holds.  Then we have
\begin{align*}
N(P)
=\SS(P_0)\mathfrak{I}P^{n-3} 
&+
O\big(M^{\frac{4n}{n+16}-1} (\rho P)^{\frac{32}{n+16}+n-4+\ve}\big)
+O\big(M^{-\frac{\psi n}{4}}\rho^n P^{n+\ve}\big)\\
&+O\big(
M^{\frac{n(n+12)}{8(n+16)}+\frac{7}{2}+\frac{11}{n-2}}
\rho^{n-\frac{1}{2}-\frac{4}{n+16}} 
P^{n-3-\frac{4}{n+16}} \big)\\
&+O\big(
M^{\frac{n}{8}-1-\frac{n(n-16)}{8(n+16)}}
\rho^{n-3-\frac{n-16}{n+16}}
P^{n-3-\frac{n-16}{n+16}+\ve}\big).
\end{align*}
\end{pro}

We have one major task remaining: we must establish an effective lower bound
for the truncated singular series \eqref{eq:truncS}.  This is probably
the most challenging part of our argument. 
The following  result will be  proved in \S \ref{s:ss}.

\begin{pro}\label{p:sum}
Let $\ve>0$ and assume that $n\geq 17$. 
Suppose that \eqref{altA} does not hold and that $C$ is $\psi$-good. 
If $\psi=\infty$ then 
$$
\SS(P_0) \gg M^{-\frac{12n}{n-9}-\ve}P_0^{-\ve} -
M^{\frac{n}{8}}P_0^{2-\frac{n}{8}+\ve}.
$$
If $\psi<\infty$ and $\delta$ satisfies
\begin{equation}
  \label{eq:del}
2<\delta <\frac{n}{8}, \quad \frac{2n}{n-8\del}< 1+2\psi,
\end{equation}
with $P_0\ll M^{1+2\psi}$, then we have 
$$
\SS(P_0) \gg M^{-\frac{12n}{n-9}-\ve} P_0^{-\ve}-
M^{\frac{n}{n-8\delta }+\ve}
P_0^{2-\delta+\ve}.
$$
\end{pro}

We now have everything in place to deduce Theorems
\ref{main-ns} and \ref{main}. Beginning with the former we suppose
that $C$ defines a hypersurface with $n\geq 17$ variables and at most isolated ordinary
singularities. Then we have already seen that $C$ is $\infty$-good
by the work of Hooley \cite[Lemma~28]{hooley2}.  
We will take 
$$
P=M^{e_3(n)+\ve}
$$
in Proposition~\ref{major'}, where $e_3(n)$ is given by \eqref{eq:bd}.
In particular it is clear that \eqref{eq:**} holds for $n \geq 17$.
With this choice of $P$, 
and for $\rho$ in the range
\eqref{eq:hope'},  we may combine the lower bound in \eqref{eq:lowerI} with
the first part of  Proposition \ref{p:sum} to conclude that
$$
\SS(P_0)\mathfrak{I}P^{n-3}\gg 
M^{(n-3)e_3(n)
-\frac{2n^2-1}{n-2}-\frac{12n}{n-9}-\ve}
$$
in Proposition~\ref{major'}. 
We must now check that each 
error term in Proposition \ref{major'} is smaller than this
bound with the above choice of $P$. 
This is clearly trivial for the second term, and for the 
the remaining terms it follows from a tedious calculation.
Thus we may conclude that $N(P)>0$ for our choice of $P$, which
therefore gives the statement of Theorem \ref{main-ns}.

Turning to the deduction of Theorem \ref{main}, in which 
no restrictions are made upon the singular locus of $C=0$, 
we observe that for a given cubic form in $n>17$ variables we can always set 
$n-17$ of the variables equal to zero in order to obtain a cubic form in
exactly $17$ variables. Thus it follows that we may proceed under the
assumption that $n=17$.
Following the argument above we will take 
$$
P=M^{E}
$$
in Propositions \ref{major'} and \ref{p:sum} and we will optimise for
$E$. In particular, taking $\rho$ in the range \eqref{eq:hope'},
we see that \eqref{eq:P0} becomes
$$
P_0=c M^{\frac{8E}{33}-\frac{5}{99}}.
$$
for a suitable absolute constant $c>0$.
We will assume that $E\geq 15$, so that \eqref{eq:**} holds.
Taking $n=17$ we note that if $C$ fails to be $\psi$-good,
then Proposition~\ref{p:altA} gives
\begin{equation}
  \label{eq:1}
\Lambda_n(C)\ll 
M^{143+\frac{1}{2}+136\psi},
\end{equation}
with an absolute implied constant. We may suppose therefore that $C$
is $\psi$-good.  

Under this assumption for $n=17$ we deduce that either \eqref{eq:1} holds
or else we can combine Propositions \ref{major'} and \ref{p:sum} with
\eqref{eq:lowerI} to obtain
\begin{align*}
N(P)
=\SS(P_0)\mathfrak{I}P^{14} 
&+
O\big(M^{\frac{461 E}{33}-\frac{3122}{99}+\ve}+
M^{17E-\frac{119}{3}-\frac{17\psi}{4}}\big),
\end{align*}
for $E\geq 15$, with
$$
\SS(P_0)\mathfrak{I}P^{14}\gg 
M^{14E -\frac{577}{15}-\ve}\big(M^{-\frac{51}{2}}-
M^{\frac{17}{17-8\delta }+
(2-\delta)({\frac{8E}{33}-\frac{5}{99}})}
\big).
$$
The latter lower bound is for any $\delta$ such that \eqref{eq:del}
holds and is valid provided that 
$P_0\ll M^{1+2\psi}$, which forces upon us the additional constraint 
\begin{equation}
  \label{eq:quatre}
\frac{8E}{33}-\frac{5}{99}\leq 1+2\psi.
\end{equation}
Let us write $\delta=2+\delta_0$, say. Then the constraints in
\eqref{eq:del} become
\begin{equation}
  \label{eq:deux}
0<\delta_0 <\frac{1}{8}, \quad \frac{34}{1-8\del_0}< 1+2\psi.
\end{equation}
We observe that the first term dominates the second term in our lower
bound for $\SS(P_0)\mathfrak{I}P^{14}$ provided that 
$$
\frac{8E}{33}-\frac{5}{99}> \frac{1}{\delta_0}
\Big(\frac{51}{2}+\frac{17}{1-8\delta_0}\Big).
$$
We will choose $\delta_0$ to minimise the right hand side,
subject to the left hand condition in
\eqref{eq:deux}. Taking $\delta_0=0.076$ and selecting $E$ to be the
least integer satisfying this inequality, we therefore deduce that 
$N(P)>0$ if 
\begin{equation}
  \label{eq:trois}
3739=E<
\frac{17\psi}{12}-\frac{81}{10}.
\end{equation}
Collecting together \eqref{eq:quatre}, \eqref{eq:deux} and
\eqref{eq:trois}, we conclude that 
$\Lambda_n(C)\ll M^{E}$ provided that $C$ is $\psi$-good for
$\psi\geq 2646$. Note that the implied constant in this estimate is independent of $n$
since we are applying the circle method machinery at $n=17$.
Taking $\psi= 2646$ and combining this with 
\eqref{eq:1}, we therefore arrive at the statement of Theorem \ref{main}.

\section{Elementary considerations}\label{s:geom}

In this section we establish Proposition \ref{p:altA}, thereby
clearing the way for an application of the circle method.
The following well-known result will prove useful, its proof being
readily supplied by consulting \cite[Lemma I.1]{schmidt'}, for example. 

\begin{lemma}\label{pigeon}
Let $m> k\geq 1$, and 
suppose that $\ma{a}_1,\ldots,\ma{a}_k\in\Z^m$ are non-zero
with modulus at most $A$. Then there exists 
$\x\in\Z^m$ such that
$$
\ma{a}_1.\x=\cdots=\ma{a}_k.\x=0
$$
and $0<|\x|\ll_m A^{\frac{k}{m-k}}$.
\end{lemma}

Our proof of Proposition \ref{p:altA} closely follows the argument of Davenport
\cite[\S 2]{dav-16}, and so we will attempt to be brief. Let us write 
\begin{equation}
  \label{eq:bil}
B_j(\x;\y) = \sum_{i,k} c_{ijk}x_iy_k,
\end{equation}
for the $j$th  bilinear form in the system $M(\x)\y$.
For given $\psi\geq 1$ we must deal with the possibility that $C$ fails to
be $\psi$-good. Thus there exists $r\in \Z\cap [0,n]$ and $H\in \Z\cap
[1,M^\psi]$ such that 
there are $\gg H^{n-r+\ve}$ points $\x$,
with $|\x| \leq H$, such that the bilinear equations $M(\x)\y=\ma{0}$ 
have exactly $r$ linearly independent solutions 
in $\y$; that is, for which the matrix $M(\x)$ has rank exactly $n-r$.  
In particular we may assume that $r\geq 1$, since there are 
$O(H^{n})$ integer vectors $\x$ with $|\x|\leq H$. 

Our goal is to derive the existence of a solution $\x\in\Z^n$ to the equation 
$C(\x)=0$, with 
\begin{equation}
  \label{eq:annA}
 0<|\x|\ll M^{\frac{n^2}{2}-1}H^{\frac{n(n-1)}{2}}.
\end{equation}
Given  that $H\leq M^\psi$ this will therefore ensure that 
\eqref{altA} holds, as required to conclude the proof of the proposition. 

Let $\XX$ denote the set of integer points $|\x|\leq H$ for which some
particular minor of order $n-r$ is non-zero and 
all minors of order $n-r+1$ are zero. Then 
$\#\XX\gg H^{n-r+\ve}$, by assumption. 
For any $\x\in\XX$, we suppose without loss of generality that 
the non-zero minor of order $n-r$, $\D$ say,  lies in the top
left-hand corner. It follows that solutions to the entire system
of equations  $M(\x)\y=\ma{0}$ can be deduced from solutions to the
first $n-r$ bilinear equations
$$
B_1(\x;\y)=\cdots=B_{n-r}(\x;\y)=0.
$$
For $1\leq i\leq r$, let $\Delta_j^{(i)}$ denote the determinant
obtained from $\Delta$ by replacing the $j$th column by the
$(n-r+i)$th column. Then an application of Cramer's rule reveals that
$r$ linearly independent solutions of the system $M(\x)\y=\ma{0}$ are given by
\begin{align*}
\y^{(1)}=&(\Delta_1^{(1)},\ldots,\Delta_{n-r}^{(1)},-\D,0,\ldots,0),\\
&\hspace{-0.2cm}\vdots\\
\y^{(r)}=&(\Delta_1^{(r)},\ldots,\D_{n-r}^{(r)},0,\ldots,0,-\D).
\end{align*}
Note that each such vector is non-zero since $\D\neq 0$,
and furthermore, has modulus $O((MH)^{n-r})$.
Indeed, each $\D_j^{(i)},\D$ is a form in $\x$ of degree $n-r$
with integer coefficients of modulus $O(M^{n-r})$.

Now consider any point $\ma{Y}=\sum_{p=1}^r \la_{p} \y^{(p)}$ in the $r$-dimensional linear space 
spanned by $\y^{(1)}, \ldots, \y^{(r)}$. Arguing as in the proof of
\cite[Lemma 3]{dav-16} we are led to the conclusion that
\begin{equation}
  \label{eq:9}
  \sum_{j}\sum_k c_{\nu j k}Y_jY_k+\sum_k \sum_{p=1}^r\la_p \D_{k}^{(p)}
\frac{\partial Y_k}{\partial x_\nu} = \sum_j\sum_{p=1}^r \la_p Y_j
\frac{\partial \D_{j}^{(p)}}{\partial x_\nu},
\end{equation}
for each $1\leq \nu\leq n$.
We now appeal to \cite[Lemma~2]{dav-16},  
with $f_1, \ldots, f_N$ being all the minors $\D_{j}^{(p)}$ of order
$n-r+1$, for $1\leq j \leq n$ and $1\leq p\leq r$. Thus there is an
element $\x \in \mathcal{X}$ for which all $\D_{j}^{(p)}$ are zero and
for which the matrix
$$
\Big\{\frac{\partial \D_{j}^{(p)}}{\partial
  x_\nu}\Big\}_{\substack{1\leq j,\nu\leq n\\1\leq p\leq r}}
$$
has rank at most $r-1$. But then, for $1\leq j \leq n$ and $1\leq p\leq
r$,  the rows
$
\frac{\partial\D_{j}^{(p)}}{\partial  x_1}  
, \ldots, \frac{\partial \D_{j}^{(p)}}{\partial  x_n}  ,
$
are all linearly dependent on $r-1$ particular rows, which we denote
by 
$
U_1^{(\rho)}, \ldots,U_n^{(\rho)},
$
for $1\leq \rho\leq r-1$. These are all integers and are the values of
bihomogeneous forms of degree $n-r+1$ in the coefficients of $C$ and degree
$n-r$ in $\x$. It therefore follows that 
$U_\nu^{(\rho)}\ll M^{n-r+1}H^{n-r}$ for $1\leq \nu \leq n$ and $1\leq \rho\leq r-1$.

We may now deduce the existence of numbers $T_{j p \rho}\in
\Q$ such that
$$
\frac{\partial\D_{j}^{(p)}}{\partial  x_\nu}=\sum_{\rho=1}^{r-1}   
T_{j p \rho}U_\nu^{(\rho)},
$$
for $1\leq j \leq n$ and $1\leq p\leq r$.
At the particular point $\x$ under consideration we make this
substitution into \eqref{eq:9}, multiply by $Y_\nu$ and finally sum over
$\nu$. This yields
\[ 
C(\ma{Y}) = \sum_{\rho=1}^{r-1}V_\rho  \sum_{q=1}^r
 \sum_\nu \lambda_q y_\nu^{(q)}U_{\nu}^{(\rho)},
\] 
where $V_\rho=\sum_{p=1}^r \sum_j \la_pY_j T_{jp\rho}$.
We will choose the numbers $\lambda_1, \ldots, \lambda_r$ to satisfy
\[ 
\sum_{q=1}^r \lambda_q \sum_{\nu} y_\nu^{(q)} U_{\nu}^{(\rho)} =0, 
\]
for $1 \leq \rho \leq r-1$. Observe that
$
\Big|\sum_\nu y_\nu^{(q)} U_{\nu}^{(\rho)}\Big|\ll M^{2(n-r)+1}H^{2(n-r)}.
$
Hence an application of Lemma \ref{pigeon} with
$(k,m)=(r-1,r)$ implies that there exists a non-zero 
solution $(\lambda_1, \ldots, \lambda_r)\in\Z^r$  to this system of 
equations, in which 
$$
\la_i\ll (M^{2(n-r)+1}H^{2(n-r)})^{\frac{r-1}{r-(r-1)}} =M^{(2(n-r)+1)(r-1)}H^{2(n-r)(r-1)},
$$
for $1\leq i\leq r$. 
But then 
\begin{align*}
|\ma{Y}|\ll \max_{i,p}|\la_i||\y^{(p)}|
&\ll
M^{(2(n-r)+1)(r-1)}H^{2(n-r)(r-1)} (MH)^{n-r}\\
&=
M^{(2r-1)(n-r)+r-1}H^{(2r-1)(n-r)}.
\end{align*}
Since the maximum over $1\leq r\leq n$ is attained at $r=\frac{n}{2}$, this
therefore confirms the existence of a solution $\x \in \Z^n$ to the
equation $C(\x)=0$ with \eqref{eq:annA} holding.

\section{The minor arcs}\label{s:minor}

Our main task in this section is to investigate the cubic exponential sum
$S(\al)$, for typical $\al\in [0,1]$, under the assumption that the
underlying cubic form is $\psi$-good.
Let $Q\geq 1$. By Dirichlet's approximation
theorem there exist coprime integers $0\leq a< q \leq Q$ such that
$\al=\frac{a}{q}+z$ for some $z\in \R$ such that $|z|\leq (qQ)^{-1}$.
Our work will be optimised by taking 
\begin{equation}
  \label{eq:Q}
  Q=M^{\frac{1}{2}} (\rho P)^{\frac{3}{2}},
\end{equation}
in which we recall the notation $M=\|C\|$ and the standing assumption
that $\rho P \geq 1$ in the definition of $S(\al).$ 

With this choice of $\al$ we must produce an upper bound for $S(\al)$
that is uniform in the various parameters $\rho,\z$ and $M$.
The basic underlying approach is that of Davenport, which is based on 
an application of Weyl differencing.

\begin{lemma}\lab{lem:Sbirch}
Let $\ve>0$ and assume that $\rho P \geq 1$. Assume
that $\al=\frac{a}{q}+z$ for coprime integers $0\leq a<q$ and that $C$ is $\psi$-good. Then we have
$$
S(\al)
\ll 
(\rho P)^{n+\ve}
\Big(\frac{1}{\rho^2P^2}+Mq|z|+\frac{q}{\rho^3 P^3}
+\frac{1}{q}\min\big\{M,\frac{1}{|z|\rho^3 P^3}\big\}+\frac{1}{M^{2\psi}}\Big)^{\frac{n}{8}}.
$$
\end{lemma}

\begin{proof}
Beginning with the  first step in the Weyl differencing process, we 
obtain the inequality
$$
|S(\alpha)|^2\leq \sum_{\w\in \Z^n\cap P\B}\Big|\sum_{\x\in \Z^n\cap \mathcal{R}(\w)}
e\big(\al (C(\x+\w)-C(\x))\big)\Big|,
$$
where $\mathcal{R}(\w)$ is a certain box inside $P\B$, depending on $\w$.
An application of Cauchy's inequality now yields
$$
|S(\alpha)|^4\ll \rho^nP^n\sum_{\w,\x\in \Z^n\cap P\B}\Big|\sum_{\y\in
  \Z^n\cap \mathcal{S}(\w,\x)}e\big(\alpha C(\w,\x;\y)\big)\Big|,
$$
where $\mathcal{S}(\w,\x)\subseteq P\B$ is a further region depending
on $\w$ and $\x$, and 
$$
  C(\w,\x;\y)=C(\w+\x+\y)-C(\w+\y)-C(\x+\y)+C(\y).
$$
Here we have used the fact that 
$
\#(\Z^n\cap P\B) \ll (\rho P+1)^n\ll \rho^n P^n,
$
which follows from the fact  that $\rho P\geq 1$.

Recall the notation introduced in \eqref{eq:C} for the coefficients of
$C$, and the definition \eqref{eq:bil} of the bilinear forms $B_i(\w;\x)$,
for $1\leq i\leq n$. It is now straightforward to arrive at the 
conclusion that 
\begin{align*}
|S(\alpha)|^4
&\ll \rho^n P^{n} \sum_{\w,\x\in \Z^n\cap P\B} \prod_{i=1}^n
\min\{\rho P, \| 6\alpha B_i(\w;\x)\|^{-1}\}.
\end{align*}
We proceed to estimate the sum over $\w$ and $\x$, which we call $M(\al,P)$.

For fixed $\w$, let $N(\w)$ denote the number of points $\x\in
\Z^n$ such that $|\x|\leq 2\rho P$ and $\|6\al B_i(\w;\x)\| < \rho^{-1}P^{-1}$, for 
$1 \leq i \leq n$. Then for any integers
$r_1,\ldots,r_n$ with $0 \le r_i < \rho P$, and fixed $\w$, we claim that there are at most $N(\w)$
points $\x\in\Z^n\cap P\B$ which satisfy
\[ 
\frac {r_i}{\rho P} \leq \{6\al B_i(\w;\x)\} < \frac{r_i+1}{\rho P},
\]
for $1\leq i\leq n$.  To see this, 
suppose that $\x_0$ is any one such
point and write $\y=\x_0-\x$. If $\x\in \Z^n\cap P\B$ also satisfies
the inequalities involving the $r_i$, then clearly 
\[ 
\| 6 \al B_i(\w;\y) \|=\| 6 \al B_i(\w;\x_0 - \x) \| < \rho^{-1}P^{-1}, 
\]
and $|\y|\leq 2\rho P$. Hence there are at most $N(\w)$ possibilities
for $\y$. 

It therefore follows that 
\begin{align*}
M(\al,P)
& \ll \sum_{\w\in \Z^n\cap P\B} N(\w) \sum _{r_1=0} ^{\rho P} \cdots \sum _{r_n=0} ^{\rho P}
          \prod_{i=1}^n\min\Big\{\rho P,\frac{\rho P}{r_i}\Big\} \\
& \ll \sum_{\w\in \Z^n\cap P\B} N(\w) (\rho P \log P)^n,
\end{align*}
whence 
\begin{align*}
|S(\alpha)|^4
&\ll (\rho P)^{2n}(\log P)^n 
\#\Big\{(\w,\x)\in\Z^{2n}: \begin{array}{l}
\w\in P\B, ~|\x|\leq 2\rho P, \\
\|6\al B_i(\w;\x)\| < (\rho P)^{-1}
\end{array}
\Big\}.
\end{align*}
Assuming that the set whose size needs to be estimated has at least one point 
$(\w_0,\x_0)$, say, we make the substitution $\ma{t}=\w_0-\w$. Then, 
since $\|\al+\be\|\leq \|\al\|+\|\be\|$ for any
real numbers $\al,\be$,  we easily deduce that 
\begin{align*}
|S(\alpha)|^4
&\ll (\rho P)^{2n}(\log P)^n 
\#\Big\{(\t,\x)\in\Z^{2n}: 
\begin{array}{l}
|\t|,|\x|\leq 2\rho P, \\
\|6\al B_i(\t;\x)\| < 2(\rho P)^{-1}
\end{array}
\Big\}.
\end{align*}

One now follows more or less verbatim the argument described in detail by Heath-Brown
\cite[\S 2]{14}. Thus we obtain
\begin{align*}
|S(\alpha)|^4
&\ll \frac{(\rho P)^{2n}(\log P)^n }{Z^{2n}}
\#\big\{(\w,\x)\in\Z^{2n}: 
|\w|,|\x|\ll Z \rho P, ~M(\w)\x=\ma{0}
\big\},
\end{align*}
in the notation of \eqref{Mdef}, for any $Z\in \R$ such that 
$$
0<Z<1,\quad Z^2\ll (Mq|z|\rho^2P^2)^{-1},\quad Z^2\ll \rho q^{-1}P,
$$
and
$$
Z^2\ll \max\Big\{\frac{q}{M\rho^2 P^2}\,,\,q\rho P|z|\Big\}.
$$
For such $Z$, we wish to apply \eqref{eq:altA'} to estimate
the above cardinality.

Now if $Z\rho P<1$ it trivially follows that
$$
|S(\al)|^4\ll (\rho P)^{4n}\ll Z^{-n}(\rho P)^{3n+\ve}.
$$
Assuming that $Z\rho P\geq 1$ we write $H=Z\rho P$. Under the
hypothesis that $C$ is $\psi$-good, we deduce that
\begin{align*}
|S(\alpha)|^4
&\ll \frac{(\rho P)^{2n}(\log P)^n }{Z^{2n}}
\sum_{1\leq r\leq n}\sum_{\substack{|\w|\ll H\\ r(\w)=r}}
(Z\rho P)^{n-r}\ll 
Z^{-n}(\rho P)^{3n+\ve},
\end{align*}
by \eqref{eq:altA'}, provided that 
$H=Z \rho P \leq M^{\psi}.$
Choosing $Z$ as big as possible, given all of these constraints, we
easily conclude the proof of  Lemma \ref{lem:Sbirch}.
\end{proof}

A useful feature of Lemma \ref{lem:Sbirch} is that the upper bound is
completely independent of the choice of $\z$ made in the definition
of the box $\B$.  We are now in a position to deduce a number of
useful estimates from this result. 
A key ingredient in our treatment of the truncated singular
series is an estimate for the complete sum
\begin{equation}
  \label{eq:complete}
  S(a,q)=\sum_{\ma{r}\bmod q} e_q(aC(\ma{r})),
\end{equation}
for given coprime integers $a,q $ such that $q\geq 1$. 
It is easily checked that the proof of Lemma \ref{lem:Sbirch} goes
through with $\B$ replaced by the box 
$[0,\rho)^n$.
Taking $(\rho,P,z)=(1,q,0)$, we deduce the
subsequent estimate for $S(a,q)$.

\begin{lemma}\label{lem:complete}
Let $\ve>0$ and assume that $C$ is $\psi$-good. Then we have 
$$
S(a,q)
\ll 
M^{\frac{n}{8}+\ve} q^{\frac{7n}{8}+\ve}+
M^{-\frac{\psi n}{4}}q^{n+\ve}.
$$
\end{lemma}

Assuming for the moment that $C$ is $\infty$-good, one can 
adapt the  van der Corput argument developed by
Heath-Brown \cite[\S 3]{14}  to derive a bound of the shape 
$$
S(a,q)
\ll 
M^{\frac{n}{6}} q^{\frac{5n}{6}+\ve}
$$
for $q\geq M$. While this is considerably sharper than Lemma
\ref{lem:complete}, it is not  enough to give 
a worthwhile saving without
an extensive overhaul of our minor arc treatment.

Recall the definition \eqref{eq:Q} of
$Q$. The following result is a further 
easy consequence of Lemma \ref{lem:Sbirch}.

\begin{lemma}\label{lem:weyl}
Let $\ve>0$ and assume that $C$ is $\psi$-good. Assume
that $\al=\frac{a}{q}+z$ for coprime integers $0\leq a<q \leq Q$ such that
$|z|\leq (qQ)^{-1}$. Then we have 
$$
S(\al)
\ll 
(\rho P)^{n+\ve}
\Big\{
q^{-\frac{n}{8}}
\min\Big\{M,\frac{1}{|z|\rho^3 P^3}\Big\}^{\frac{n}{8}}
+M^{-\frac{\psi n}{4}}\Big\}.
$$
\end{lemma}

We are now equipped to tackle the proof of Proposition \ref{minor}.
Let $\al\in \m$, and write $\al=\frac{a}{q}+z$
for coprime integers $a,q$ such that $0\leq a< q \leq Q$, and 
$z\in \R$ such that $|z|\leq (qQ)^{-1}$. Here, as usual, $Q$ is given
by   \eqref{eq:Q}.  In particular, we may assume
that $q>P_0$ whenever $|z|\leq M^{-1}\rho^{-3}P^{-3}P_0=\frac{P_0}{Q^2}$, else $\al \in \M(a,q)$. 
It now follows from Lemma \ref{lem:weyl}
that 
\begin{align*}
\int_\m |S(\al)| \d\al 
&\ll (\rho P)^{n+\ve}\Big(
M^{-\frac{\psi n}{4}}+ 
 \sum_{q\leq Q} 
 q^{1-\frac{n}{8}}\int_{-\frac{1}{qQ}}^{\frac{1}{qQ}}
\min\Big\{M,\frac{1}{|z|\rho^3 P^3}\Big\}^{\frac{n}{8}} \d z\Big)\\
&\ll  (\rho P)^{n+\ve}\Big(M^{-\frac{\psi n}{4}}+ 
M^{\frac{n}{8}-1}\rho^{-3}P^{-3}P_0^{2-\frac{n}{8}}\Big)\\
&= M^{-\frac{\psi n}{4}}\rho^n P^{n+\ve} +
M^{\frac{n}{8}-1} \rho^{n-3} P^{n-3+\ve}P_0^{2-\frac{n}{8}},
\end{align*}
since $n\geq 17$. This therefore concludes 
the proof of Proposition \ref{minor}.

\section{The major arcs}\label{s:major}

The purpose of this section is to establish
Proposition \ref{major}, for which we assume that $n\geq 17$.
It will be convenient to define $B$ to be the smallest real number
exceeding $1$ such that 
$\B\subseteq [-B,B]^n$. In particular we have 
\begin{equation}
  \label{eq:B}
1\leq B\ll 1+|\z|.
\end{equation}
Let $a,q$ be coprime integers such that $0\leq a<q\leq P_0$,  and let
$\al=\frac{a}{q}+z\in \M(a,q)$. Then we have
\begin{equation}
  \label{eq:hope}
S(\al)=\sum_{\ma{r}\bmod{q}}e_q(aC(\ma{r}))\sum_{\substack{\y\in\Z^n\\
    \ma{r}+q\y\in
   P\B}} e(zC(\ma{r}+q\y)).
\end{equation}
We wish to show that the sum over $\x$ can be replaced by an
integral.  It turns out that a sharper error term is available through the Poisson
summation formula, rather than using the approach adopted by
Davenport \cite[\S 8]{dav-16}. This is achieved in the following
result.

\begin{lemma}\label{lem:pois}
Let $\al=\frac{a}{q}+z\in \M(a,q)$, for 
coprime integers $a,q$ such that $0\leq a<q\leq P_0$, where $P_0$ is
given by \eqref{eq:P0}. Assume that \eqref{eq:**} holds.
Then either $\Lambda_n(C)=O(1)$, or else we have 
$$
S(\al)=q^{-n}S(a,q)I_P(z)
+O\big(q (\rho P)^{n-1}\big),
$$
where $S(a,q)$ is given by \eqref{eq:complete} and 
$$
I_P(z)=\int_{P\B} e(zC(\x))\d\x.
$$
\end{lemma}

\begin{proof}
Let $f \in
\Z[X_1,\ldots,X_n]$ be a homogeneous polynomial  of degree $d \geq 3$
for which one of the partial derivatives 
$\frac{\partial^2 f}{\partial X_i^2}$ vanishes identically. It is then a trivial
matter to see that the coefficient of $X_i^d$ vanishes in $f$, whence
the equation $f=0$ has the non-trivial integer solution
$(0,\ldots,0,1,0\ldots,0)$, with only the $i$th
component of the vector being non-zero.

Let us define a box $\mathcal{C}=\prod_{i=1}^n [a_i,b_i]$, for 
$a_i,b_i\in \R$. We let
$
R_{\CC}=\max_{1\leq i \leq n}|b_i-a_i|.
$
Still in the setting of arbitrary forms $f$ as above, and with 
$\lambda$ being an arbitrary non-zero real number, we will show that
\begin{equation}
  \label{eq:cloud}
  \sum_{\x \in \Z^n\cap P\mathcal{C}} e(\lambda f(\x)) =
  \int_{P\mathcal{C}} e(\lambda f(\t)) \d\t+O_d(\psi^{-1} (R_\CC P)^{n-1}),
\end{equation}
under the hypothesis that none of the 
partial derivatives 
$\frac{\partial^2 f}{\partial X_i^2}$ vanish identically, and furthermore,
there exists $\psi\in (0,1]$ such that
$$
|\la \nabla f(\x)| \leq 1-\psi
$$
for all $\x \in P\mathcal{C}$.  
We may clearly proceed under the assumption that $R_\CC P \geq 1$,
since otherwise \eqref{eq:cloud} is trivial.

We argue by induction on $n$. 
For the case $n=1$, with $I=[Pa_1,Pb_1]$, we deduce from
\cite[Proposition 8.7]{ik} that
\begin{equation}
  \label{eq:8.7}
 \sum_{x \in \Z\cap I} e(\lambda f(x)) =
  \int_{I} e(\lambda f(t)) \d t+O(\psi^{-1}),
\end{equation}
provided that $f''(t)\neq 0$ and $|\la f'(t)| \leq 1-\psi$ on
$I$. When $f\in\Z[X]$ is a polynomial, then provided 
it does not vanish identically, we see that $f''$
is a polynomial that has at most $d-2$ roots in $\R$. 
Breaking up the interval $I$ into the finite number of pieces on which
$f''$ has constant sign, we easily conclude that
\eqref{eq:cloud} holds in the case $n=1$.

When $n>1$ we have 
$$
\sum_{\x \in \Z^n\cap P\mathcal{C}} e(\lambda
f(\x)) =
 \sum_{\y \in \Z^{n-1}\cap P\mathcal{C}_1} 
\sum_{x_n \in \Z\cap I} e(\lambda f_1(x_n)),
$$
where $\y=(x_1,\ldots,x_{n-1})$, 
$\CC_1=\prod_{i=1}^{n-1} [a_i,b_i]$, $I=[Pa_n,Pb_n]$, and 
$f_1(X)=f(\y,X)$.
We may now apply \eqref{eq:8.7} to estimate the
inner sum, finding that 
$$
\sum_{\x \in \Z^n\cap P\mathcal{C}} e(\lambda
f(\x)) =
 \sum_{\y \in \Z^{n-1}\cap P\mathcal{C}_1} 
\Big(\int_{I} e(\lambda f_1(t))\d t +O(\delta_{f_1} \psi^{-1})\Big),
$$
where 
$$
\delta_{f_1}=\min\big\{R_\CC P, \#\{t \in I: f_1''(t)=0\}\big\}.
$$
Let us divide the $\y$ in the outer summation into
two sets: those for which 
$f''(\y,X)$ vanishes identically
as a polynomial in $X$, and those for which it does not. It follows
from elimination theory that the cardinality of the first set is
$O((R_\CC P)^{n-2})$, since $\frac{\partial^2 f}{\partial X_n^2}$ doesn't
vanish identically. The cardinality of the second set is clearly 
$O((R_\CC P)^{n-1})$. Moreover, for 
$\y$ belonging to the second set we have
$\delta_{f_1}=O_d(1)$. 
Putting this together we deduce that
$$
\sum_{\x \in \Z^n\cap P\mathcal{C}} e(\lambda
f(\x)) =
 \int_{I}  \sum_{\y \in \Z^{n-1}\cap
   P\mathcal{C}_1}e(\lambda f(\y,t)) \d t 
+O_d\big(\psi^{-1}(R_\CC P)^{n-1})\big).
$$
Finally, we apply the induction hypothesis to estimate the integrand,
which thereby completes the proof of \eqref{eq:cloud}.

We are now ready to establish Lemma \ref{lem:pois}. 
Let $f(\y)=C(\ma{r}+q\y)$.
We may assume that
none of the diagonal second order derivatives of $f$ vanish
identically, since the alternative hypothesis implies that $\Lambda_n(C)=O(1)$.
Given any $\y$ in the box $P\CC$ determined by the inequalities 
$\ma{r}+q\y \in P\B$, we clearly have 
$
z \nabla f(\y)
\ll q |z|M (BP)^2,
$
where $B$ satisfies the inequality in \eqref{eq:B} and $\rho,|\z|$
are assumed to satisfy \eqref{eq:hope'}. 
On recalling the definition of the major arcs, together with the
expression \eqref{eq:P0} for $P_0$, so it follows that
\begin{align*}
q |z|M (BP)^2\ll 
\frac{(1+|\z|)^2 P_0^2}{\rho^3 P}&\ll 
\frac{ M^{\frac{2n}{n+16}+\frac{2}{n-2}} (\rho P)^{\frac{16}{n+16}}}{\rho^3
  P}\\
&\ll 
 M^{\frac{2n}{n+16}+\frac{2}{n-2}+\frac{(3n+32)(2n+1)}{(n+16)(n-2)}} P^{\frac{16}{n+16}-1}  \\
&=
\Big(\frac{ M^{\frac{8n^2+65n+64}{n-2}}}{P^{n}
  }\Big)^{\frac{1}{n+16}}.
\end{align*}
Taking the inequality in \eqref{eq:**} for $P$ one easily deduces that
this is $O(M^{-\delta})$ for a certain value of $\delta>0$ when $n\geq
17$. Taking $\psi=1-q |z|M (BP)^2$, we therefore deduce that $\psi\gg 1$.
In the present setting we have $R_\CC=\frac{2\rho}{q}$. 
It therefore follows from 
applying \eqref{eq:cloud} in \eqref{eq:hope} that
\begin{align*}
S(\al)-
q^{-n}S(a,q)I_P(z) 
&\ll 
q^n 
\Big(\frac{\rho P}{q}\Big)^{n-1}=
q(\rho P)^{n-1}.
\end{align*}
This completes the proof of the lemma.
\end{proof}

It is perhaps interesting to compare Lemma \ref{lem:pois} with
Davenport's approach. If we were to apply his argument directly we would
instead be led to an overall error term
$
O\big(q B^2\rho^{n-3}P^{n-1}P_0\big)
$
in the lemma,  which is visibly worse.

Integrating over $|z|\leq (M\rho^3P^3)^{-1}P_0=ZP^{-3}$, say, we
may now deduce from Lemma~\ref{lem:pois} that
\begin{align*}
\int_{\M(a,q)}S(\al)\d \al
&=q^{-n}S(a,q)\int_{|z|\leq ZP^{-3}}I_P(z)\d z +
O\Big(
\frac{q \rho^{n-4}P^{n-4}P_0}{M}\Big)\\
&=q^{-n}S(a,q)\mathfrak{I}(Z)P^{n-3} +
O\Big(
\frac{\rho^{n-4}P^{n-4}P_0^2}{M}\Big),
\end{align*}
where
$$
\mathfrak{I}(Z)
=\int_{|z|\leq Z}I_1(z)\d z.
$$
On recalling the definition \eqref{eq:truncS} of $\SS(P_0)$, and
summing over the relevant $a,q$, we deduce that 
\begin{equation}\label{eq:anna1}
\int_{\M}S(\al)\d \al
=\SS(P_0)\mathfrak{I}(Z)P^{n-3} +
O\Big(
\frac{\rho^{n-4}P^{n-4}P_0^4}{M}\Big),
\end{equation}
where $Z=M^{-1}\rho^{-3}P_0$.

We would now like to show that $\mathfrak{I}(Z)$
can be approximated by a positive constant as $Z\rightarrow \infty$.
Thus it is time to choose the real point $\z$ that features in our
definition of $\mathcal{B}=\mathcal{B}(\z;\rho)$. The following
argument is a refinement of an analogous result due to Lloyd \cite[Lemma
4.1]{lloyd}.

\begin{lemma}\label{lem:4.1}
Either 
\begin{equation}\label{shoe}
\Lambda_n(C)\ll M^{\frac{1}{n-2}},
\end{equation}
or else there exist constants $c,c'>0$ and a vector
$\z=(\xi,\y)\in \R^n$ such that $C(\z)=0$ and  
\begin{equation}
  \label{eq:bound_z}
M^{-1-\frac{2}{n-2}}\ll|\xi|\ll M^{\frac{1}{n-2}},
\quad |\y|\ll M^{\frac{1}{n-2}},
\end{equation}
with  
\begin{equation}
  \label{eq:c1}
\partial_1=\frac{\partial C}{\partial x_1} (\xi,\y) > \frac{c}{M^{1+\frac{4}{n-2}}}
\end{equation}
and
\begin{equation}
  \label{eq:c2}
|\partial_2|=\Big|\frac{\partial C}{\partial x_2} (\xi,\y)\Big| > \frac{c'}{M^{2+\frac{7}{n-2}}}.
\end{equation}
\end{lemma}

\begin{proof}
Let us write
$$
C(\x)=ax_1^3+F_1 x_1^2+ F_2x_1+F_3,
$$
where $a=c_{111}\in \Z$ and $F_i\in \Z[X_2,\ldots,X_n]$ are forms
of degree $i$,  with $F_1=3\sum_{i=2}^n c_{11i}X_i$. It follows from
\eqref{eq:anna} that $a\gg M$. Furthermore, 
Lemma~\ref{pigeon} implies 
that there exists $\y=(y_2,\ldots,y_n)\in \Z^{n-1}$
such that $F_1(\y)=0$ and 
$
0<|\y|\ll M^{\frac{1}{n-2}}.
$
In particular we may assume that $F_3(\y)\neq 0$, else consideration
of the point $(0,\y)$ shows that \eqref{shoe} holds.

Writing $F_2=F_2(\y)$ and $F_3=F_3(\y)$ we have 
$C(X_1,\y)=aX_1^3+F_2X_1+F_3$, with
$$
\frac{\partial C}{\partial X_1}(X_1,\y)=3aX_1^2+ F_2.
$$
Since $a>0$ we have $C(x_1,\y)\rightarrow \pm \infty$ 
 as $x_1\rightarrow \pm \infty$, and so
 there exists $\xi\in \R$ such that $C(\xi,\y)=0$ and $\xi F_3<0.$
We have  $\xi(a\xi^2+F_2)=-F_3$, whence
\begin{equation}
  \label{eq:hard}
  a\xi^2+F_2=-\frac{F_3}{\xi}=\left|\frac{F_3}{\xi}\right|>0.
\end{equation}
It will be convenient to note that $|F_3|\geq 1$. Our argument
now breaks into two cases, according to whether $F_2$ is non-negative or negative.

Suppose that $F_2\geq 0$. Then \eqref{eq:hard} yields
$a|\xi^3|\leq |F_3| \ll M^{1+\frac{3}{n-2}}$ and $|\xi|\leq |\frac{F_3}{a}|^{\frac{1}{3}}$. 
But then it follows that $\xi\ll M^{\frac{1}{n-2}}$ and 
\begin{align*}
|\xi|\geq (a\xi^2+F_2)^{-1}&\gg M^{-(1+\frac{2}{n-2})}.
\end{align*}
In particular $\z=(\xi,\y)$ satisfies \eqref{eq:bound_z}.
When $F_2<0$ we deduce from \eqref{eq:hard} that
$|\xi|\geq |\frac{F_3}{a}|^{\frac{1}{3}}\gg M^{-\frac{1}{3}}$. Furthermore it follows that
$$
a\xi^2+F_2\leq |F_3||a^{-1}F_3|^{-\frac{1}{3}}=|a|^{\frac{1}{3}}|F_3|^{\frac{2}{3}},
$$ 
whence
$
\xi^2\leq |\frac{F_2}{a}|+|\frac{F_3}{a}|^{\frac{2}{3}}\ll M^{\frac{2}{n-2}}.
$
This establishes that $\xi\ll M^{\frac{1}{n-2}}$, which is enough to
show that $\z=(\xi,\y)$ satisfies \eqref{eq:bound_z} in this case too.

Let $\partial_i=\frac{\partial C}{\partial x_i}(\z)$, for $1\leq i \leq n$.
We therefore deduce from  \eqref{eq:hard} that
$$
\partial_1=3a\xi^2+F_2 \geq 2a \xi^2 \gg M^{-1-\frac{4}{n-2}},
$$
as required for \eqref{eq:c1}. In order to establish \eqref{eq:c2} we
employ the triangle inequality in Euler's identity $\x.\nabla
C(\x)=3C(\x)$. Thus it follows that
$$
|y_2\partial_2+\cdots +y_n\partial_n| \geq |\xi \partial_1| -3
|C(\xi,\y)|=|\xi \partial_1|\gg M^{-2-\frac{6}{n-2}},
$$
whence there exists $i \in \{2,\ldots, n\}$ such that 
$$
M^{\frac{1}{n-2}}|\partial_i|\gg |y_i \partial_i| \gg
M^{-2-\frac{6}{n-2}}.
$$
This therefore completes the proof of Lemma \ref{lem:4.1}, possibly
after relabelling the variables of $C$.
\end{proof}

We may clearly proceed under the assumption that \eqref{shoe} does not
hold in Lemma~\ref{lem:4.1}. 
Now for any $Z\geq 1$ we have 
$$
\mathfrak{I}(Z)=
\int_{-Z}^Z\int_{\mcal{B}}e(\theta
C(\ma{x}))\d\x\d\theta
=
\int_{\mcal{B}_\rho}\frac{\sin 2\pi Z C(\ma{z}+\w)}{\pi C(\z+\w)}\d\w,
$$
where $\mathcal{B}_\rho=\{\w\in\R^n: |\w|<\rho\}$.
The idea now is to make the change of variables $t=C(\z+\w)$, using
this relation to express $w_1$ in terms of $t$ and
$\tilde{\w}=(w_2,\ldots,w_n)$.
We begin by noting that
$$
\frac{\partial C}{\partial x_1} (\z+\w)=
\frac{\partial C}{\partial x_1} (\z)+
O(\rho M|\z|)
$$
for any $\w\in \mcal{B}_\rho$. Hence it follows from the choice of
$\z$ made in Lemma \ref{lem:4.1}, and in particular
\eqref{eq:bound_z} and \eqref{eq:c1}, that
\begin{equation}
  \label{eq:c1..}
\frac{1}{ M^{1+\frac{4}{n-2}}}\ll 
\frac{\partial C}{\partial x_1} (\z+\w)\ll 
M|\z|^2+ \rho M |\z|\ll M^{1+\frac{2}{n-2}},
\end{equation}
provided that $\rho$ satisfies \eqref{eq:hope'}.

Let $f:\R^n\rightarrow \R^n$ be the transformation taking $\w$ to
$(t,\tilde{\w})$, with $t=C(\z+\w)$. Write
$\mathcal{R}=f(\mathcal{B}_\rho)$. Since $f$ has polynomials for
components, so $f$ is differentiable on
$\mathcal{B}_\rho$. Furthermore it is a bijection between
$\mcal{B}_\rho$ and $\mcal{R}$, since $\frac{\partial C}{\partial x_1} (\z+\w)>0$
for any $\w\in \mcal{B}_\rho$ by \eqref{eq:c1..}.
Define the function $g:\mathcal{R}\rightarrow \mcal{B}_\rho$ by 
$$
f^{-1}(t,\tilde{\w})=(w_1,\tilde{\w})=\big(g(t,\tilde{\w}),\tilde{\w}\big).
$$
By definition $w_1=g(t,\tilde{\w})$ is the inverse of $t=C(\z+\w)$,
regarded as functions of $t$ and $w_1$ only. Thus 
$$
g_1(t,\tilde{\w})=\frac{\partial g}{\partial t}(t,\tilde{\w})=
\left(\frac{\partial C}{\partial x_1}(\z+\w_g) \right)^{-1}
$$
is the Jacobian  of $f^{-1}$, where
$\w_g=(g(t,\tilde{\w}),\tilde{\w})$.
We deduce from \eqref{eq:c1..} that 
\begin{equation}
  \label{eq:c1...}
\frac{1}{ M^{1+\frac{2}{n-2}}}\ll  g_1(t,\tilde{\w})\ll M^{1+\frac{4}{n-2}},
\end{equation}
for any $(t,\tilde{\w})\in \mcal{R}$. 

Making the change of variables from $w_1$ to $t$ in our expression for
$\mathfrak{I}(Z)$, we
find that
$$
\mathfrak{I}(Z)=
\int_{\mcal{R}}
\frac{\sin 2\pi Z t}{\pi t}  g_1(t,\tilde{\w}) \d t \d\tilde{\w}
=
\int_{-\sigma}^{\sigma}
\frac{\sin 2\pi Z t}{\pi t}  V(t)\d t, 
$$
where 
\begin{equation}
  \label{eq:sigg}
\sigma=\sup\{|C(\z+\w)|: \w \in \mathcal{B}_\rho\}  
\end{equation}
and 
\begin{equation}
  \label{eq:Vt}
V(t)=\int_{\substack{\tilde{\w}\in \R^{n-1}\\ (t,\tilde{\w})\in
    \mathcal{R}}} 
g_1(t,\tilde{\w})\d\tilde{\w}.
\end{equation}
Assuming that $V$ is well-behaved, the Fourier inversion theorem leads
to the equality $\lim_{Z\rightarrow \infty} \mathfrak{I}(Z)=V(0)$. We
will need a more explicit version of this, for 
which a careful analysis of 
the function $V(t)$ is required. This analysis is routine but lengthy,
it being necessary to establish the continuity of $V(t)$,
together with the existence of left and right derivatives at all
points in the interval $(-\sigma,\sigma)$. One also requires an upper bound for
the size of these derivatives. This calculation is carried out in full
detail by Lloyd \cite[Lemma~4.7 and Lemma 4.8]{lloyd} and we content
ourselves with recording the outcome of his investigation in the
following result.

\begin{lemma}\label{int_lloyd}
Let $Z\geq 1$. Then we  have 
$$
\mathfrak{I}(Z)
=V(0)+ O\big(V(0)\sigma^{-1}Z^{-1}+A(1+\sigma)Z^{-\frac{1}{2}}
\big),
$$
where $\sigma$ is given by \eqref{eq:sigg}, $V(0)$ is given by \eqref{eq:Vt} and 
$$
A=\rho^{n-2}\partial_1^{-1}(|\partial_2|^{-1}+\rho
M|\z|\partial_1^{-2}).
$$
\end{lemma}

The integral $V(0)$ is over a box in $\R^{n-1}$ with side
length $2\rho$ and it follows from \eqref{eq:c1...}  that 
\begin{equation}
  \label{eq:V0}
\frac{\rho^{n-1}}{M^{1+\frac{2}{n-2}}}\ll V(0) \ll \rho^{n-1} M^{1+\frac{4}{n-2}}. 
\end{equation}
Next, we note that in 
view of \eqref{eq:hope'}, \eqref{eq:c1} and \eqref{eq:c2},
we clearly have 
\begin{align*}
A
&\ll \rho^{n-2}\partial_1^{-1}(|\partial_2|^{-1}+
M^{-1-\frac{4}{n-2}}\partial_1^{-2})\ll M^{3+\frac{11}{n-2}}\rho^{n-2}.
\end{align*}
Turning to $\sigma$, as given by \eqref{eq:sigg}, we observe that
for any $\w \in \mathcal{B}_\rho$ we have
$
C(\z+\w)=\w.\nabla C(\z)+O(\rho^2 M |\z|).
$
In particular, for $\z,\rho$ satisfying \eqref{eq:hope'}, we deduce
that $\sigma\ll 1$ and $\sigma^{-1}\ll \rho^{-1}\partial_1^{-1}\ll \rho^{-1}M^{1+\frac{4}{n-2}}$.

Putting this all together we deduce from Lemma \ref{int_lloyd} that
$$
\mathfrak{I}(Z)
=V(0)+ O\big(\rho^{n-2}M^{3+\frac{11}{n-2}}Z^{-\frac{1}{2}}\big),
$$
for $Z\geq 1$, where $V(0)$ satisfies \eqref{eq:V0}.
Inserting this into 
\eqref{eq:anna1}, with $Z=M^{-1}\rho^{-3}P_0$, we therefore conclude that 
\begin{align*}
\int_{\M}S(\al)\d \al
=\SS(P_0)V(0)P^{n-3} &+
O\Big(
\frac{\rho^{n-4}P^{n-4}P_0^4}{M}\Big)\\
&+O\big(|\SS(P_0)|M^{\frac{7}{2}+\frac{11}{n-2}}
\rho^{n-\frac{1}{2}} 
P^{n-3}P_0^{-\frac{1}{2}} \big).
\end{align*}
Since $n\geq 17$ we may deduce from Lemma \ref{lem:complete} that
\begin{align*}
\SS(P_0)
&\ll 
\sum_{q\leq P_0}
\Big(M^{\frac{n}{8}} q^{1-\frac{n}{8}+\ve}+
M^{-\frac{\psi n}{4}} q^{1+\ve}\Big)
\ll M^{\frac{n}{8}}+
M^{-\frac{\psi n}{4}} P_0^{2+\ve}.
\end{align*}
Taking $\mathfrak{I}=V(0)$ and noting that
$
M^{\frac{7}{2}+\frac{11}{n-2}}\rho^{-\frac{1}{2}}\ll M^6
$
for $n\geq 17$, we 
therefore arrive at the statement of  Proposition \ref{major}.

\section{Cubic forms over finite fields}

We are now ready to prove Theorem \ref{enterprise}.
To this end we make use of the $h$-invariant $h=h(C)$ of $C$ over
$\F_p$ as introduced by Davenport and Lewis \cite{DL}. If $h \ge 8$,
then an easy application of their work gives
the result. In fact it produces an asymptotic formula for $\rho(p)$,
rather than merely a lower bound. Thus we may assume that $h\leq 7$. Without loss of generality
we suppose that $C$ is of the form
\[
  C(X_1, \ldots, X_n) = \sum_{i=1}^h X_i Q_i(X_1, \ldots, X_n)
\]
for suitable quadratic forms $Q_i$, where $h \le 7$. 
We will achieve our aim by fixing choices of $X_1,\ldots,X_h$
which leave the resulting polynomial with a quadratic part
of sufficiently large rank. This will allow us to apply the
following elementary result.

\begin{lemma}
\label{gauss}
Let $Q \in \F_p[X_1, \ldots, X_n]$ be a quadratic
form of rank at least three, let 
$L \in \F_p[X_1, \ldots, X_n]$ be a linear form and let $c \in
\F_p$. 
Then we have
\[
\#\{\x\in \F_p^n: Q(\x)+L(\x)+c\equiv 0 \bmod{p}\}=p^{n-1} + O(p^{n-2}).
\]
\end{lemma}

\begin{proof}
When $L$ and $c$ vanish the estimate is well-known and can be proved using the
explicit evaluation of the quadratic Gaussian sum. The general case
follows on considering the 
form $Q(\mathbf{X}) + ZL(\mathbf{X}) + cZ^2$ in $n+1$ variables and counting the solutions to
the polynomial congruence projectively.
\end{proof}

In the above decomposition of $C$, let
\[
  \hat{Q}_i(X_{h+1}, \ldots, X_n) = Q_i(0, \ldots, 0, X_{h+1}, \ldots,
  X_n), \quad (1 \le i \le h).
\]
If $X_{h+1}^2$ does not
occur in any of the $Q_i$, then we may conclude that $C$ is of the form
\begin{equation}
\label{spock}
X_{h+1}Q(X_1, \ldots, X_h, X_{h+2}, \ldots, X_n)
 + \tilde{C}(X_1, \ldots, X_h, X_{h+2}, \ldots, X_n)
\end{equation}
for a suitable quadratic form $Q$ and cubic form $\tilde{C}$. By
the non-degeneracy of $C$, the form $Q$ is not identically zero. Hence there
are $p^{n-1} + O(p^{n-2})$ choices for $x_1, \ldots, x_h, x_{h+2},
\ldots, x_n \in \F_p$ such that $Q(x_1, \ldots, x_h, x_{h+2}, \ldots, x_n)
\ne 0$. In each case we can solve (\ref{spock}) for $x_{h+1}$ to find
a zero $\mathbf{x} \in \F_p^n$ of $C$. We conclude that 
(\ref{kirk}) is true in this case.

We now assume without loss of generality that $X_{h+1}^2$
occurs in $Q_1$, say. Then by completing the square and using a suitable
non-singular linear transformation on the variables $X_{h+1}, \ldots, X_n$ we
can assume that $\hat{Q}_1$ is of the form
\[
  a_1 X_{h+1}^2 +
  \tilde{Q}_1(X_{h+2}, \ldots, X_n),
\]
for $a_1 \ne 0$ and a suitable quadratic form $\tilde{Q}_1$. We continue
by considering $X_{h+2}^2$. If $X_{h+2}^2$ does not occur in any $Q_i$,
then we can use the same argument as above to deduce the lower bound
(\ref{kirk}). Alternatively, there are two further cases to consider
according to whether or not $X_{h+2}^2$ occurs in $\hat{Q}_1$.
If it does then we can again complete
the square to obtain
\[
  \hat{Q}_1(X_{h+1}, \ldots, X_n) = a_1 X_{h+1}^2 + a_2 X_{h+2}^2 +
  \tilde{\tilde{Q}}_1(X_{h+3}, \ldots, X_n)
\]
where $a_1a_2 \ne 0$ and $\tilde{\tilde{Q}}_1$ is a suitable quadratic form.
In the remaining case, we may suppose without loss of generality $X_{h+2}^2$ occurs in
$\hat{Q}_2$, say, giving 
\[
  \hat{Q}_2(0, X_{h+2}, \ldots, X_n) = b_2 X_{h+2}^2 +
  \tilde{Q}_2(X_{h+3}, \ldots, X_n),
\]
for $b_2 \ne 0$ and a suitable quadratic form $\tilde{Q}_2$. In a similar
fashion we repeat the analysis on $X_{h+3}$. This leads to another diagonal term for
$\hat{Q}_1$ or $\hat{Q}_2$, or one for $\hat{Q}_3$.

Let
\[
  D(x_1, \ldots, x_h) = \det \left(
  \sum_{i=1}^h x_i Q_i(0, \ldots, 0, X_{h+1}, X_{h+2}, X_{h+3}, 0, \ldots,
  0) \right),
\]
the determinant being that of a quadratic form in
$X_{h+1}, X_{h+2}, X_{h+3}$. We claim that 
$D$ is not identically zero. To see this, suppose first that  
$\hat{Q}_1$ splits off three diagonal
terms $a_1 X_{h+1}^2 + a_2 X_{h+2}^2 + a_3 X_{h+3}^2$, where
$a_1 a_2 a_3 \ne 0$, then we can choose $x_1=1$ and $x_2 = \cdots = x_h = 0$
to get a non-singular quadratic form in $X_{h+1}, X_{h+2},
X_{h+3}$. Suppose next that $\hat{Q}_1$ splits off a pair of  diagonal terms $a_1 X_{h+1}^2
+ a_2 X_{h+2}^2$, where $a_1 a_2 \ne 0$, and $\hat{Q}_2$ splits off
$a_3 X_{h+3}^2$, where $a_3 \ne 0$. Our construction
implies that $\hat{Q}_1(X_{h+1}, X_{h+2}, X_{h+3}, 0, \ldots, 0)$ has no term in
$X_{h+3}$, 
so that $D(x_1,x_2,0, \ldots, 0)$ is the determinant of 
\[
  x_1(a_1 X_{h+1}^2 + a_2 X_{h+2}^2)+x_2(a_3 X_{h+3}^2 +\tilde{\tilde{Q_2}}(X_{h+1}, X_{h+2})),
\]
for a suitable binary quadratic form $\tilde{\tilde{Q_2}}$.
It is now clear that  we can choose
$x_1, x_2$ to
arrive at a non-singular quadratic form in $X_{h+1}, X_{h+2}, X_{h+3}$. 
The remaining cases are handled in a similar way.

For fixed $x_1, \ldots, x_h$ such that $D(x_1, \ldots, x_h) \ne 0$, 
an application of Lemma \ref{gauss} reveals that there are $p^{n-h-1}
+ O(p^{n-h-2})$  solutions of
\[
  \sum_{i=1}^h x_i Q_i(x_1, \ldots, x_n) = 0
\]
in $x_{h+1}, \ldots, x_n$. Moreover, since $D$ is not identically zero,
there are  $p^h + O(p^{h-1})$ possible choices for $x_1, \ldots,
x_h$. This completes the proof of (\ref{kirk}).

\section{The truncated singular series}\label{s:ss}

In this section we establish Proposition \ref{p:sum}, for  which
we may freely assume that $n\geq 17$. In particular the equation $C=0$ has a
non-singular solution in $\Q_p^n$.
This fact has been established by a number of authors, but a
comprehensive treatment can be found in Davenport \cite[Chapter
18]{dav}, where the quantity 
$\D(C) \in \Z$ is introduced. This is defined to be the greatest common divisor
of all the $n\times n$ subdeterminants of the $n\times
\frac{1}{2}n(n+1)$ matrix formed from the 
coefficients of $C$. It is left invariant under any
unimodular change of variables and it is easy to see that $\D(C)=0$ if
and only if $C$ is degenerate. In our work we 
may assume that $\D(C)\neq 0$, for otherwise 
Lloyd's work \cite[Lemma 5.9]{lloyd} shows that 
$\Lambda_n(C)\ll M^{n-1}$, which is certainly sharper than the bound
in \eqref{altA}.  It follows that 
$$
0<\D(C)\ll M^n.
$$
Recall the definitions \eqref{eq:rho-p}, \eqref{eq:truncS} of
$\varrho(p^k)$ and  $\SS(R)$, respectively. 
We will write $\varrho^*(p^k)$ for the set of non-singular
solutions modulo $p^k$.
As is well-known, we
have 
$$
\sum_{i=0}^k A(p^i)= 
p^{-k(n-1)}\varrho(p^k),
$$
where
\begin{equation}
  \label{eq:A}
A(q) = 
\sum_{\substack{0\leq a<q\\ \hcf(a,q)=1}}
q^{-n}S(a,q)
\end{equation}
and $S(a,q)$ is given by \eqref{eq:complete}.
With this notation we have $\SS(P_0)=\sum_{q\leq P_0}A(q)$.

Our first task is to produce some good lower bounds for $\rho(p^k)$
that are uniform in the coefficients of $M$. 
The following result is pivotal in our work and is based on Theorem~\ref{enterprise}.

\begin{lemma}\label{lem:E}
Assume that $p\nmid \D(C)$ 
and $p \gg 1$, with  $n\geq 10$. Then for any
$k \geq 1$ we have 
$$
\rho^*(p^k)\geq p^{k(n-1)}\big(1+O(p^{-1})\big).
$$
\end{lemma}
\begin{proof}
Since $p \nmid \D(C)$, the cubic form $C$ is non-degenerate
modulo $p$. It will suffice to assume that $k=1$ in the statement of
the lemma, the general case following from Hensel's lemma. 
Now if there are no singular zeros counted by $\rho(p)$, then
Theorem~\ref{enterprise} 
implies that there is nothing to do. Alternatively,
we may assume without loss of
generality that $(1, 0, \ldots, 0)$ is a singular zero of
$C$. Thus 
\begin{equation}
\label{warp}
  C(X_1, \ldots, X_n) = X_1 Q(X_2, \ldots, X_n) + \tilde{C}(X_2,
  \ldots, X_n),
\end{equation}
for a quadratic form $Q$ and a cubic form $\tilde{C}$. Since $C$ is
non-degenerate, $Q$ does not vanish identically. Hence there are
$p^{n-1} + O(p^{n-2})$ choices for $x_2, \ldots, x_n \in \F_p$ such that
$Q(x_2, \ldots, x_n) \ne 0$. Each choice gives a non-singular zero
$\mathbf{x} \in \F_p^n$ of $C$ by solving (\ref{warp}) for $x_1$ and
noting that
\[
  \frac{\partial C}{\partial X_1} (\mathbf{x}) = Q(x_2, \ldots, x_n)
  \ne 0.
\]
This therefore shows that 
$\rho^*(p)\geq p^{n-1}+O(p^{n-2})$, as required to complete the proof
of the lemma. 
\end{proof}

The order of $C$ modulo a prime $p$ is defined to be the positive
integer $h$ such that the reduction of $C$ modulo $p$  is a form in
precisely $h$ variables, with no non-singular linear transformation
modulo $p$ taking the cubic form into a form in fewer than $h$
variables. When $p\nmid \D(C)$ one has $h=n$, for example.  
The following result is weaker than Lemma \ref{lem:E}, but has the
advantage of applying to forms of small order.

\begin{lemma}\label{lem:CW}
Assume that $C$ has order $h$ modulo $p$, with $h\geq 4$. Then for any
$k \geq 1$ we have 
$$
\rho^*(p^k)\geq p^{k(n-1)}\big(1+O(p^{-\frac{1}{2}})\big).
$$
\end{lemma}

\begin{proof}
In view of Hensel's lemma it will suffice to establish the inequality
for $k=1$. 
Since $h\geq 4$ an application of the Chevalley--Warning theorem
implies that the congruence 
\begin{equation}
  \label{eq:cong}
  C(X_1, \ldots, X_n) \equiv 0 \pmod p
\end{equation}
has a non-trivial solution. We now apply
the work of Leep and Yeomans \cite[Lemma~3.3]{LY}: if a
form of prime degree has a non-trivial zero modulo $p$ and is
non-degenerate, then either the form is absolutely irreducible
modulo $p$ or it is reducible modulo $p$. In the first case,
we can apply the Lang--Weil estimate in order to deduce that
\[
\rho^*(p)= p^{n-1} + O(p^{n-\frac{3}{2}}),
\]
which is satisfactory.

In the second case, we may suppose that 
\[
  C(X_1, \ldots, X_n) \equiv L(X_1, \ldots, X_h)
  Q(X_1, \ldots, X_h) \pmod p
\]
for a suitable linear form $L$ and quadratic form $Q$.
Assuming without loss of generality that the coefficient of $X_1$ in
$L$ is non-zero, we may make the non-singular linear transformation
$Y_1=L(X_1,\ldots,X_h)$ and $Y_i=X_i$ for $2\leq i\leq h$. Thus $C$
can be taken to be $Y_1 Q(Y_1, \ldots, Y_h)$ modulo $p$, 
for a suitable quadratic form $Q$. 
Moreover, we may assume that $Q(0,Y_2, \ldots, Y_h)$ 
is not identically zero, since otherwise we could carry out a
non-singular linear change of variables bringing $C$ into a form
with fewer than $h$ variables present. 
Each $\mathbf{x} \in(\Z/p\Z)^h$ with $p\mid x_1$
and $p\nmid Q(0,x_2, \ldots,
x_h)$ leads to a non-singular zero of \eqref{eq:cong}.
Hence 
\[
\rho^*(p)\geq 
  p^{n-1}-2p^{n-2},
\]
since there $\leq 2p^{n-2}$ choices of $x_2, \ldots, x_n$
modulo $p$ such that $p\mid Q(0,x_2,\ldots,x_h)$.
This completes the proof of the lemma. 
\end{proof}

For any prime $p$ it follows from 
\cite[Lemma 18.7]{dav} that $C$ satisfies property
$\mathcal{A}(p^{\ell(p)})$ for some positive integer $\ell(p)\leq 3m(p)+3$, where
$m(p)$ satisfies $p^{m(p)(n-9)}\mid \D(C)$. Here property
$\mathcal{A}(p^{\ell(p)})$ 
means that the congruence $C\equiv 0
\bmod{p^{2\ell(p)-1}}$ has a solution $\x$ modulo $p^{2\ell(p)-1}$ 
for which $p^{\ell(p)-1}\| \nabla C(\x)$. We will associate to each
prime $p$ the number 
\begin{equation}
  \label{eq:kp}
  k(p)=
\begin{cases}
\max_{t\in \N: p^t\leq P_0}\{t\}, & \mbox{if $p\nmid \D$,}\\
\max_{t\in \N: p^t\leq P_0}\{t, 2\ell(p)-1\}, & \mbox{if $p\mid \D$.}
\end{cases}
\end{equation}
We may now define the truncated Euler product 
$$
S(P_0)=\prod_{p\leq P_0} \sum_{i=0}^{k(p)} A(p^i)=
\prod_{p\leq P_0} p^{-k(p)(n-1)}\rho(p^{k(p)}).
$$
The following result is concerned with a uniform lower bound for this
quantity.

\begin{lemma}\label{lem:prod}
Let $\ve>0$. Then we have
$$
S(P_0)\gg M^{-\frac{12n}{n-9}-\ve}P_0^{-\ve}.
$$
\end{lemma}

\begin{proof}
We break the product $S(P_0)$ into those primes which divide $\D(C)$ and
those which do not. Beginning with the latter, it follows from Lemma
\ref{lem:E} and Merten's formula that
\begin{align*}
\prod_{\substack{1\ll p\leq P_0\\ p\nmid \D(C)}} p^{-k(p)(n-1)}\rho(p^{k(p)})
&\geq \prod_{\substack{1\ll p\leq P_0\\ p\nmid \D(C)}}
\big(1+O(p^{-1})\big)
\gg  M^{-\ve}P_0^{-\ve}, 
\end{align*}
for any $\ve>0$.  To deal with the primes $p\ll 1$ such that $p\nmid
\D(C)$ we note that $\rho^*(p)\geq 1$ for such primes, whence 
\begin{align*}
\prod_{\substack{p\ll 1 \\ p\nmid \D(C)}} p^{-k(p)(n-1)}\rho(p^{k(p)})
&\gg 1,
\end{align*}
by a lifting argument.

Turning to the contribution from primes $p\mid \D(C)$, we note that
the lower bound is trivial when $\D(C)=1$, for then 
the product is empty. Thus we assume that $\D(C)>1$.
We have two basic
possibilities: either $h\geq 4$, where $h$ is the order of $C$ modulo
$p$, or $h<4$. In the former case it follows from Lemma \ref{lem:CW}
that 
\begin{equation}
  \label{eq:lift-off'}
\rho(p^{k(p)})\gg p^{k(p)(n-1)},
\end{equation}
for $p\gg 1.$ But the same estimate holds for $p\ll 1$ since
$\rho^*(p)\geq 1$ in this setting by \cite[Lemma~18.3]{dav}.
In the second case we  write
$$
C(X_1,\ldots,X_n)\equiv C_1(X_1,\ldots,X_h) \pmod{p},
$$
for a further cubic form $C_1$ that cannot be expressed in fewer than
$h$ variables after any non-singular linear transformation modulo $p$.
We now have two further cases according
to whether or not $C_1$ has a non-trivial zero modulo $p$.

Suppose first that $C_1$ has a non-trivial zero modulo $p$, which we may assume to be 
$(1,0,\ldots,0)$. If this zero is non-singular, then any zero
$$
(1,0,\ldots,0,x_{h+1},\ldots,x_n)
$$
of $C$ will be non-singular modulo $p$. In this way we obtain
$\rho^*(p)\geq p^{n-h}\geq p^{n-3}$, giving
\begin{equation}
  \label{eq:lift-off}
\rho(p^{k(p)})\geq p^{k(p)(n-1)-2}.
\end{equation}
Alternatively, if $(1,0,\ldots,0)$ is a singular zero of $C_1$ modulo
$p$, then we have
\[
  C_1(X_1, \ldots, X_h) \equiv X_1 Q(X_2, \ldots, X_h)
  + C_2(X_2, \ldots, X_h) \pmod p,
\]
where $Q$ is a quadratic form and $C_2$ is a cubic form. Choose $x_2,
\ldots, x_h$ modulo $p$ such that $Q(x_2, \ldots, x_h) \not \equiv 0
\bmod p$. Note that $Q$ cannot be identically zero modulo $p$, since 
otherwise $C_1$ would have order less than $h$. Defining 
\[
  x_1 \equiv Q(x_2, \ldots, x_h)^{-1} C_2(x_2, \ldots, x_h) \pmod p,
\]
and allowing $x_{h+1}, \ldots, x_n$ to be arbitrary modulo $p$, we
therefore conclude that $\rho^*(p)\geq p^{n-h}
\ge p^{n-3}$. Hence \eqref{eq:lift-off} holds in this case too.

Finally we must deal with the possibility that $h<4$ and $C_1$ has no non-trivial zero modulo $p$.
But then each solution of $C(X_1, \ldots, X_n) \equiv 0 \bmod p$ forces
$p\mid x_i$ for $1\leq i \leq h$, whereas $x_{h+1},
\ldots, x_n$ are free. We conclude that
\[
  \varrho(p^k) = p^{n-h} \varrho_1(p^{k-1}), 
\]
for each $k \in \N$, where 
$\varrho_1$ is defined as for $\varrho$ but with $C$ replaced by the
cubic form 
$$
\tilde{C}(X_1,\ldots,X_n)=p^{-1}C(pX_1,\ldots,pX_h,X_{h+1},\ldots,X_n). 
$$
It is clear that one can now repeat the above argument with 
$C$ replaced by $\tilde{C}$. 
Since $C$ has property $\mathcal{A}(p^{\ell(p)})$, there exists a $p$-adic zero of $C$
with gradient divisible by at most $\ell(p)-1$ powers of $p$. Thus 
after at most $\ell(p)-1$ steps we are
in a situation where a non-singular zero modulo $p$ exists. 

Combining \eqref{eq:lift-off'} and \eqref{eq:lift-off}, a modest
pause for thought therefore leads to the final outcome
that
$$
\varrho(p^{k(p)})
  \gg 
p^{k(p)(n-1)-2\ell(p)}
\geq 
p^{k(p)(n-1)-6(m(p)+1)}.
$$
Here the final inequality follows from the fact that 
$\ell(p)\leq 3m(p)+3$, as recorded above. 
Recalling that $m(p)$ satisfies $p^{m(p)(n-9)}\mid \D(C)$, 
we therefore deduce that there exists a
constant $c'\geq 1$ such that 
\begin{align*}
\prod_{\substack{p\leq T\\ p \mid \D(C)}} 
p^{-k(p)(n-1)}\rho(p^{k(p)})
\geq
  \prod_{p \mid \D(C)} \frac{1}{c'p^{6(m(p)+1)}}
 &\ge {c'}^{-\omega(\D(C))}\D(C)^{-\frac{6(m(p)+1)}{m(p)(n-9)}}\\
&\gg M^{-\frac{12n}{n-9}-\ve},
\end{align*}
for all $\ve>0$. This completes the proof of the lemma.
\end{proof}

Our final task is to show how to approximate our truncated singular
series $\SS(P_0)$ by the truncated product $S(P_0)$. 
For a given prime $p$ recall the definition \eqref{eq:kp} of $k(p)$.
Define 
$$
\mcal{Q}(P_0)=\{q\in \N: q>P_0, ~ p^i\mid q \Rightarrow \mbox{$p\leq
  P_0$ and $i\leq k(p)$}\}.
$$
We now have everything in place to analyse the 
difference
$$
R(P_0)=|\SS(P_0)-S(P_0)| \leq \sum_{q\in \mcal{Q}(P_0)} 
|A(q)|,
$$
where $A(q)$ is given by \eqref{eq:A}.
The following result will allow us to conclude the first part of
Proposition \ref{p:sum}. 

\begin{lemma}\label{R-inf}
Assume that $C$ is $\infty$-good. Then we have 
$$
R(P_0)\ll  M^{\frac{n}{8}}P_0^{2-\frac{n}{8}+\ve}.
$$
\end{lemma}

\begin{proof}
It follows from Lemma \ref{lem:complete} that
$
A(q) \ll 
M^{\frac{n}{8}} q^{1-\frac{n}{8}+\ve},
$
for any $\ve>0$. This readily establishes the lemma for
$n\geq 17$.
\end{proof}

Combining Lemma \ref{R-inf} with the lower bound for $S(P_0)$ in 
Lemma \ref{lem:prod} we are easily led to the first part of
Proposition \ref{p:sum}.  

It remains to consider the case of $\psi$-good forms, with $\psi<\infty$.
Assume that $\del$ is chosen so that  \eqref{eq:del} holds. Then we
have the following result, which once 
combined with Lemma \ref{lem:prod}
thereby establishes the second part of 
Proposition~\ref{p:sum}.

\begin{lemma}
Assume that $C$ is $\psi$-good, for $\psi<\infty$. 
Assume furthermore that $P_0\ll M^{1+2\psi}$.
Then we have 
$$
R(P_0)\ll  M^{\frac{n}{n-8\delta }+\ve}
P_0^{2-\delta+\ve}.
$$
\end{lemma}

\begin{proof}
Let $q$ be an integer in the interval
\begin{equation}
  \label{eq:interval}
M^{\frac{n}{n-8\del}+\ve} \leq q \leq M^{1+2\psi}.
\end{equation}
Then the upper bound here implies that
$
A(q) \ll 
M^{\frac{n}{8}} q^{1-\frac{n}{8}+\ve},
$
in Lemma \ref{lem:complete}. Assuming that $\del$ satisfies
\eqref{eq:del}, we therefore conclude from \eqref{eq:interval} that 
\begin{equation}
  \label{eq:AAA}
A(q)=O(q^{1-\delta}),
\end{equation}
uniformly in $M$.

In order to produce an upper bound 
$R(P_0)$ we will need to sort the $q$ according to how it factorises.
For ease of notation let us henceforth set
$$
A=\frac{n}{n-8\del}+\ve, \quad B=1+2\psi.
$$
Note that the inequalities in \eqref{eq:del} imply that $2A<B$.
We claim that for each $q\in \mcal{Q}(P_0)$
there is a factorisation
\begin{equation}
  \label{eq:fact}
q=q_1\cdots q_t q_{t+1},
\end{equation}
with $\hcf(q_i,q_j)=1$ for each $1\leq i< j\leq t+1$, such that 
$q_i$ satisfies the upper and lower bounds in \eqref{eq:interval} for
$1\leq i \leq t$ and $q_{t+1}< M^{A}$. 

Taking the claim on faith for the moment we note that for $\delta$
selected as in \eqref{eq:del}, we may apply Lemma \eqref{eq:AAA} to
deduce that
$$
|A(q)|=|A(q_1)|\cdots |A(q_t)| |A(q_{t+1})|\ll (q_1\cdots
q_t)^{1-\delta} q_{t+1}\ll
M^{A}(q_1\cdots
q_t)^{1-\delta}.
$$
Here we have employed the trivial bound $|A(q_{t+1})|\leq q_{t+1}$.
Writing $q=q_0q_{t+1}$, with $q_0=q_1\cdots q_t$, we conclude that 
\begin{align*}
R(P_0)\ll M^{A}
\sum_{q_{t+1}<M^{A}}
\sum_{q_0>\frac{P_0}{q_{t+1}}} q_0^{1-\delta+\ve}
&\ll M^{A} P_0^{2-\delta+\ve},
\end{align*}
which is satisfactory for the lemma.  Here we have used the fact that $\delta>2$
in \eqref{eq:del}.

It remains to establish the claimed factorisation \eqref{eq:fact} of
$q$.  Suppose that 
$$
q=p_1^{k_1}\cdots p_r^{k_r}
$$ 
is the factorisation of
$q$ into primes. We make the partition $\{1,\ldots,r\}=I\sqcup J$, where $j \in J$
if and only if $p_j^{k_j}$ satisfies the upper and lower bounds in
\eqref{eq:interval}.  
For $j\in J$ we may then take $p_j^{k_j}$ to form the set
$\{q_1,\ldots,q_{\#J}\}$ in \eqref{eq:fact}. Turning to the remaining
factor $q'=\prod_{i\in I}p_i^{k_i}$ of $q$, we rewrite this as
$$
q'=p_1'^{\ell_1}\cdots p_s'^{\ell_s}, \quad 
p_1'^{\ell_1}\leq \cdots \leq p_s'^{\ell_s}<M^{A}.
$$
Here the final inequality is by construction.
Suppose that
$p_1'^{\ell_1}p_2'^{\ell_2}\geq M^{A}$. Then we may
take
$q_{\#J+1}=p_1'^{\ell_1}p_2'^{\ell_2}$ in \eqref{eq:fact} since the
final inequality in \eqref{eq:del} implies that
$$
M^{A}\leq q_{\#J+1}< M^{2A}\leq
M^{B}.
$$
We may then repeat the analysis on $q_{\#J+1}^{-1}q'$. Alternatively, if 
$p_1'^{\ell_1}p_2'^{\ell_2}< M^{A}$ we ask instead
whether or not 
$p_1'^{\ell_1}p_2'^{\ell_2}p_3'^{\ell_3}$ exceeds 
$M^{A}$. If the answer is in the affirmative then
we can take 
$q_{\#J+1}=p_1'^{\ell_1}p_2'^{\ell_2}p_3'^{\ell_3}$, and if negative,
then we proceed to consider the size of 
$p_1'^{\ell_1}\cdots p_4'^{\ell_4}$. It is clear that this
process terminates and leads to the factorisation described in
\eqref{eq:fact}. 
\end{proof}

\end{document}